\documentclass{article}

\usepackage{amsmath}
\usepackage{amssymb}
\usepackage{amsthm}
\usepackage{mathrsfs}

\usepackage[usenames,dvipsnames]{color}
\usepackage[%
   breaklinks,
   colorlinks=true, 
   citecolor=blue,
   linkcolor=red,
   pagebackref=false,
   pdfauthor={},
   pdftitle={On the Spectral Decomposition of Dichotomous and Bisectorial Operators},]{hyperref}

\usepackage{tikz}
\usepackage{pstricks}
\usepackage{pst-plot}

\theoremstyle{plain}
\newtheorem{theorem}{Theorem}[section]
\newtheorem{corollary}[theorem]{Corollary}
\newtheorem{proposition}[theorem]{Proposition}
\newtheorem{lemma}[theorem]{Lemma}

\theoremstyle{definition}
\newtheorem{definition}[theorem]{Definition}

\theoremstyle{remark}
\newtheorem{example}[theorem]{Example}
\newtheorem{remark}[theorem]{Remark}

\newcommand\mylabelenumi[1]{{\upshape (#1)}}
\newcommand\enumiref[1]{\mylabelenumi{\ref{#1}}}

\newcommand{\rd}{d}

\newcommand{\set}[2]{\{#1:#2\}}
\newcommand{\bigset}[2]{\bigl\{#1:#2\bigr\}}
\newcommand{\biggset}[2]{\biggl\{#1:#2\biggr\}}

\newcommand{\pmat}[1]{\begin{pmatrix}#1\end{pmatrix}}

\newcommand{\I}{{\rm i}}
\newcommand{\e}{{\rm e}}

\newcommand{\mD}{\mathcal D}

\newcommand{\C}{\mathbb C}
\newcommand{\N}{\mathbb N}
\newcommand{\R}{\mathbb R}
\newcommand{\Z}{\mathbb Z}

\renewcommand{\phi}{\varphi}
\newcommand{\eps}{\varepsilon}

\DeclareMathOperator{\linspan}{span}
\DeclareMathOperator{\range}{Im}
\DeclareMathOperator{\diag}{diag}

\DeclareMathOperator{\re}{Re}	

\begin{document}

\title{On the Spectral Decomposition of Dichotomous and Bisectorial Operators}
\author{%
Monika Winklmeier%
\thanks{Departamento de Matem\'aticas, Universidad de los Andes, Cra. 1a No 18A-70, Bogot\'a, Colombia. mwinklme@uniandes.edu.co.},
Christian Wyss%
\thanks{Fachgruppe Mathematik und Informatik, Bergische Universit\"at Wuppertal, Gau{\ss}str. 20, 42097 Wuppertal, Germany. wyss@math.uni-wuppertal.de.}
}

\maketitle
\begin{abstract}
   \noindent
  For an unbounded operator $S$ on a Banach space
  the existence of invariant subspaces corresponding to its spectrum in
  the left and right half-plane is proved.
  The general assumption on $S$ is the
  uniform boundedness of the resolvent 
  along the imaginary axis.
  The projections associated with the invariant subspaces
  are bounded if $S$ is strictly dichotomous,
  but may be unbounded in general.
  Explicit formulas for these projections in terms of resolvent integrals
  are derived and used to obtain
  perturbation theorems for dichotomy.
  All results apply, with certain simplifications, to bisectorial operators.

  \bigskip\noindent
  \textbf{Mathematics Subject Classification (2010).}\\
  Primary 47A15; Secondary 47A10, 47A55, 47A60, 47B44.

  \smallskip\noindent
  \textbf{Keywords.}
  Bisectorial operator; dichotomous operator; unbounded projection; 
  invariant subspace; $p$-subordinate perturbation.

  \smallskip\noindent
  The final publication is available at
  \href{http://link.springer.com/article/10.1007/s00020-015-2218-5}{link.springer.com}

\end{abstract}

\section{Introduction}

Let $S$ be a densely defined linear operator $S$ on a Banach space $X$ such that
the imaginary axis belongs to the resolvent set $\varrho(S)$.
A fundamental question is whether there exist closed invariant subspaces 
$X_+$ and $X_-$ which correspond to the spectrum of $S$ in the
right and left half-plane $\C_+$ and $\C_-$, respectively.
$S$ is called \emph{dichotomous} if these subspaces exist and yield
a decomposition $X=X_+\oplus X_-$.
If in addition the restrictions $-S|X_+$ and $S|X_-$ generate
exponentially decaying semigroups, then $S$ is 
\emph{exponentially dichotomous}.
Dichotomy and exponential dichotomy have found a wide range of applications,
e.g.\ to
canonical factorisation and Wiener-Hopf integral operators
\cite{BGK-wienerhopf,BGK1986},
and to block operator matrices and Riccati equations
\cite{LT2001,LRR2002,ran-mee,tretter-wyss}.
An extensive account may be found in the monograph \cite{vdmee-book}.

The investigation of dichotomous operators was started
by H.~Bart, I.~Gohberg and M.A.~Kaashoek in \cite{BGK1986},
where they established a sufficient condition for dichotomy:
If a strip around the imaginary axis belongs to $\varrho(S)$,
the resolvent $(S-\lambda)^{-1}$ is uniformly bounded on this strip, i.e.,
\begin{equation}\label{eq:intro-bnd}
  \sup_{|\re\lambda|\leq h}\|(S-\lambda)^{-1}\|<\infty
\end{equation}
for some $h>0$, and
the integral 
\begin{equation}\label{eq:intro-int}
  Px=\frac{1}{2\pi\I}\int_{h-\I\infty}^{h+\I\infty}
  \frac{1}{\lambda^2}(S-\lambda)^{-1}S^2x\,d\lambda,
  \qquad x\in\mD(S^2),
\end{equation}
extends to a bounded operator $P$ on $X$, then $S$ is dichotomous and
$P$ is the projection onto $X_+$ along $X_-$.
On the other hand, there are simple examples where
\eqref{eq:intro-bnd} holds, $X_\pm$ exist, but
$X_+\oplus X_-\subset X$ is only dense
and $S$ is not dichotomous.

In our first main result, Theorem~\ref{theo:splitting},
we prove 
that \eqref{eq:intro-bnd} is in fact sufficient for
the existence of invariant subspaces,
even if $S$ is not dichotomous:
If \eqref{eq:intro-bnd} holds, then the subspaces
\[G_\pm=\set{x\in X}{(S-\lambda)^{-1}x\text{ has a bounded analytic 
    extension to }\overline{\C_\mp}}\]
are closed, $S$-invariant and satisfy
$\sigma(S|G_\pm)\subset\C_\pm$.
Moreover, the integral \eqref{eq:intro-int} extends to a closed,
possibly unbounded operator $P$, which is the projection onto $G_+$
along $G_-$.
Finally, $P$ is bounded if and only if $S$ is
dichotomous with respect to $X=G_+\oplus G_-$.
This decomposition
has the additional property that the resolvents of $S|G_\pm$
are uniformly bounded on $\C_\mp$, and we call $S$
\emph{strictly dichotomous} in this case.

One important class of operators satisfying \eqref{eq:intro-bnd} are
\emph{bisectorial} and \emph{almost bisectorial} operators, 
for which $\I\R\subset\varrho(S)$ and
\begin{equation}\label{eq:intro-bisect}
  \|(S-\lambda)^{-1}\|\leq\frac{M}{|\lambda|^\beta}, \qquad 
  \lambda\in\I\R\setminus\{0\},
\end{equation}
with some $M>0$ and $\beta=1$ in the bisectorial, $0<\beta<1$ in the
almost bisectorial case.
In Theorem~\ref{theo:almbisect}
we show that for bisectorial and almost bisectorial $S$ equation
\eqref{eq:intro-int} simplifies to
\begin{equation}\label{eq:intro-int2}
  Px=\frac{1}{2\pi\I}\int_{h-\I\infty}^{h+\I\infty}
  \frac{1}{\lambda}(S-\lambda)^{-1}Sx\,d\lambda,
  \qquad x\in\mD(S),
\end{equation}
and that the restrictions $S|G_\pm$ are (almost) sectorial, i.e.,
they satisfy the resolvent estimate \eqref{eq:intro-bisect} on $\C_\mp$.
The results for bisectorial $S$ were to some extent obtained 
by  W.~Arendt and A.~Zamboni in \cite{arendt-zamboni}.
In particular, they constructed closed projections 
by using a rearranged version of \eqref{eq:intro-int2};
our construction of $P$ in 
Theorem~\ref{theo:splitting} is in fact a generalisation of their 
method
to the weaker setting of \eqref{eq:intro-bnd}.
Note that bisectoriality does not imply dichotomy and hence unbounded
$P$ are still possible here as Example~\ref{ex:almbisect} shows.

In addition to the characterisation of dichotomy
in Theorems~\ref{theo:splitting} and~\ref{theo:almbisect},
we also derive perturbation results.
Theorem~\ref{theo:dichot-pert} states that 
if $S$ is strictly dichotomous and $T$ is 
another densely defined operator on $X$ such that
\begin{enumerate}
\item
  $\set{\lambda\in\C}{|\re\lambda|\leq h}\subset\varrho(S)\cap\varrho(T)$,
\item
  $\sup_{|\re\lambda|\leq h}|\lambda|^{1+\eps}
  \|(S-\lambda)^{-1}-(T-\lambda)^{-1}\|<\infty$ for some $\eps>0$, and
\item
  $\mD(S^2)\cap\mD(T^2)\subset X$ is dense,
\end{enumerate}
then $T$ is strictly dichotomous too.
A similar result was obtained in \cite{BGK1986}, but under
significantly stronger conditions, namely $\eps=1$ in (ii) and
$\mD(T^2)\subset\mD(S^2)$ instead of (iii).
It is precisely the existence of the closed projections
which allows us to use the more general condition (iii).
If in addition $S$ is (almost) bisectorial, then
$T$ is  (almost) bisectorial too and
condition (iii) may be relaxed to
\begin{enumerate}
\item[(iii$'$)] $\mD(S)\cap\mD(T)\subset X$ is dense,
\end{enumerate}
see Theorem~\ref{theo:dichot-pert-beta}.
In \cite{tretter-wyss} this result was obtained for bisectorial $S$
and $T=S+R$ where $R$ is $p$-sub\-or\-di\-nate to $S$, i.e.,
$\mD(S)\subset \mD(R)$ and
$\|Rx\|\leq c\|x\|^{1-p}\|Sx\|^p$, $x\in\mD(S)$,
with constants $c>0$ and $0\leq p<1$.
In this case (ii) holds with $\eps=1-p$ and (iii$'$) is trivially
satisfied since $\mD(S)=\mD(T)$.
On the other hand, (iii) and (iii$'$) allow for a wider
class of perturbations with $\mD(S)\neq\mD(T)$,
e.g.\ if $T=S+R$ holds only in an extrapolation space,
see Example~\ref{ex:ham}.

Most of our results remain valid for non-densely defined $S$.
In particular, this is true for the main Theorems~\ref{theo:splitting}
and~\ref{theo:almbisect}, but not for the perturbation results.
Unless explicitly stated otherwise, linear operators are are not assumed to be densely defined.

There exist different approaches to dichotomy:
For bisectorial $S$ an equivalent condition for dichotomy
in terms of complex powers of $S$ is given in  \cite{dore-venni}.
For generators of $C_0$-semigroups, exponential dichotomy
is equivalent to the hyperbolicity of the semigroup
\cite{kaashoek-verduyn-lunel},
the latter being important e.g.\  in the study of nonlinear 
evolution equations.
Finally, our approach is connected to
(bounded as well as unbounded) functional calculus.

This article is organised as follows:
In Section~\ref{sec:proj} we collect general facts about unbounded projections,
in particular Lemma~\ref{lem:Apm} on the existence of closed projections
corresponding to invariant subspaces.
Section~\ref{sec:dichot} contains the definition and basic properties
of dichotomous operators.
Here we also show that dichotomy alone does not uniquely determine the
decomposition $X=X_+\oplus X_-$ while strict dichotomy does.
In Section~\ref{sec:splitting} we derive our main Theorem~\ref{theo:splitting}
in the general setting \eqref{eq:intro-bnd}.
The case of bisectorial and almost bisectorial operators is then considered in
Section~\ref{sec:bisect}.
For such operators
we also provide some results on the 
location of their spectrum 
and derive yet another integral representation for $P$,
which was used in \cite{LT2001,tretter-wyss}.
Section~\ref{sec:Mpm} is devoted to the subtle problem that the restrictions
$S|G_\pm$ are not necessarily densely defined even if $S$ is.
Following \cite{BGK1986} we consider certain subspaces $M_\pm\subset G_\pm$
for which the parts of $S$ in $M_\pm$ are densely defined.
We derive conditions for $M_\pm=G_\pm$,
which is in turn equivalent to $S|G_\pm$ being densely defined.
The perturbation results are contained in Section~\ref{sec:pert},
and in Section~\ref{sec:examp} we provide some additional examples
to illustrate our theory.
Finally, as an application, we show the dichotomy of
a Hamiltonian operator matrix from
control theory whose off-diagonal entries map into extrapolation spaces.
In the previous results 
\cite{LRR2002,wyss-unbctrlham,tretter-wyss}
only settings without
extrapolation spaces or with the additional
assumption of a Riesz basis of eigenvectors could be handled.

We use the following notation: 
The open right and open left half-plane is denoted by $\C_+$ and $\C_-$, respectively.
If $X, Y$ are Banach spaces, then $T(X\to Y)$ denotes a linear operator from a (not necessarily dense) domain in $X$ into $Y$.
If $Z\subset X$, then we denote the restriction of $T$ to $Z$ by $T|Z$.
The set of all bounded linear operators from $X$ to $Y$ is denoted by $L(X,Y)$ and we set $L(X) = L(X,X)$.
Let $U\subset \C$.
We call a function $\phi:U\to X$ \emph{analytic} if it admits an analytic extension to some open neighbourhood of $U$.
For the argument of a complex number 
we choose the range $-\pi<\arg z\leq\pi$ for $z\neq0$ and set $\arg 0=0$.

\section{Unbounded projections} 
\label{sec:proj}

Our main tool for investigating the dichotomy of an operator $S$
are projections corresponding to invariant subspaces.
In the general case, these projections will be unbounded and
the direct sum of the corresponding subspaces is not the whole space $X$.
Unbounded projections associated with  bisectorial operators
have been studied in \cite{arendt-zamboni}.

\begin{definition}
  Let $X$ be a Banach space.
  A (possibly unbounded)
  operator $P(X\to X)$ is called a \emph{projection} if
  $\range(P)\subset \mD(P)$ and $P^2=P$.
\end{definition}

In other words, $P$ is a projection in the algebraic sense on the 
vector space $\mD(P)$.
If $P$ is a projection, then 
\begin{equation}  \label{eq:unbproj}
  \mD(P)=\range(P)\oplus\ker(P).
\end{equation}
The \emph{complementary projection} is given by $Q=I-P$, $\mD(Q)=\mD(P)$.
In this case $\range(Q)=\ker(P)$, $\ker(Q)=\range(P)$.

Conversely, if $X_1,X_2\subset X$ are linear subspaces
such that $X_1\cap X_2=\{0\}$,
then there are corresponding complementary projections $P_1$, $P_2$ with
$\mD(P_1)=\mD(P_2)=X_1\oplus X_2$, $\range(P_1)=X_1$ and $\range(P_2)=X_2$.

\begin{remark}
  \begin{enumerate}
  \item A projection $P$ is closed if and only if
    $\range(P)$ and $\ker(P)$ are closed subspaces. 
  \item A closed projection $P$ is bounded if and only if
    $\range(P)\oplus\ker(P)$ is closed.
  \end{enumerate}
\end{remark}

The next lemma gives a sufficient condition on a linear operator $S$ that allows for the construction of a pair of closed complementary projections which commute with $S$ and yield $S$-invariant subspaces.
\begin{lemma}\label{lem:Apm}
  Let $S(X\to X)$ be a closed operator with $0\in\varrho(S)$.
  Suppose that there are bounded  operators $A_1,A_2\in L(X)$ satisfying
  \begin{gather}
    A_1+A_2=S^{-2},\label{eq:Apm-sum}\\
    A_1A_2=A_2A_1=0,\label{eq:Apm-cmpl}\\
    A_jS^{-1}=S^{-1}A_j, \quad
    j=1,2.\label{eq:Apm-comm1}
  \end{gather}
  Then:
  \begin{enumerate}
  \item The operators $P_j:=S^2A_j$, $j=1,2$,
    are closed complementary projections,
    $\mD(S^2)\subset \mD(P_1)=\mD(P_2)$ and
    $P_j x=A_j S^2x$ for $x\in\mD(S^2)$.
  \item
    The subspaces $X_j=\range(P_j)$ are closed,
    $S$- and $(S-\lambda)^{-1}$-invariant for $\lambda\in\varrho(S)$
    and satisfy
    $X_1=\ker(A_2)$, $X_2=\ker(A_1)$. Moreover
    \begin{equation}\label{eq:Apm-spec}
      \sigma(S)=\sigma(S|X_1)\cup\sigma(S|X_2).
    \end{equation}
  \end{enumerate}
\end{lemma}

\begin{proof}
  (i)
  Since $S^2$ is closed and $A_j$ is bounded, $P_j$ is closed.
  Note that $x\in\mD(P_j)$ if and only if $A_j x\in\mD(S^2)$. 
  Hence \eqref{eq:Apm-sum} implies that $\mD(P_1)=\mD(P_2)$ and that
  $P_1+P_2=I$ on $\mD(P_1)$.
  From \eqref{eq:Apm-comm1} 
  it follows that
  \begin{equation}\label{eq:Apm-comm2}
    A_j Sx=SA_j x \quad\text{for}\quad x\in\mD(S),\,j=1,2.
  \end{equation}
  For $x\in\mD(P_1)$, this implies $A_2P_1x=A_2S^2A_1x=S^2A_2A_1x=0$
  and hence $\range(P_1)\subset\mD(P_j)$, $P_2 P_1=0$ and
  $P_1^2=(I-P_2)P_1=P_1$.
  Thus
  $P_1$ and $P_2$ are complementary projections.
  Additionally, \eqref{eq:Apm-comm2} yields that if $x\in\mD(S^2)$, then
  $x\in\mD(P_1)$ and
  $A_j S^2x=S^2A_j x=P_j x$.

  (ii)
  Since $S^2$ is invertible, $X_1=\ker(P_2)=\ker(A_2)$ and analogously
  $X_2=\ker(A_1)$.
  Consequently these subspaces are closed.
  By a straightforward calculation \eqref{eq:Apm-comm2} yields
  \begin{equation}\label{eq:Apm-comm3}
    A_j(S-\lambda)^{-1}=(S-\lambda)^{-1}A_j \quad\text{for all}\quad
    \lambda\in\varrho(S), j=1,2.
  \end{equation}
  This and \eqref{eq:Apm-comm2} imply that
  the spaces $X_j$ are $S$- and $(S-\lambda)^{-1}$-invariant.
  In particular it follows that $\mD(S)\cap X_j = S^{-1}(X_j)$.
  To obtain \eqref{eq:Apm-spec}, we show now the equivalent identity
  \[\varrho(S)=\varrho(S|{X_1})\cap\varrho(S|{X_2}).\]
  From the invariance properties of $X_j$ it is easily seen that
  for $\lambda\in\varrho(S)$ the operator
  $S|X_j-\lambda$ is bijective with inverse
  $(S-\lambda)^{-1}|X_j$; hence $\lambda\in\varrho(S|X_j)$.
  For the other inclusion
  let $\lambda\in\varrho(S|{X_1})\cap\varrho(S|{X_2})$.
  If $(S-\lambda) x=0$, then
  $x\in\mD(S^3)$,  $A_jS^3x=S^3A_jx$ and hence
  $x\in\mD(P_j)$, $P_jx\in\mD(S)$.
  Consequently
  \[0=(S-\lambda)x
  =(S|{X_1}-\lambda)P_1x+(S|{X_2}-\lambda)P_2x.\]
  This yields $P_1 x=P_2x=0$, i.e.\ $x=0$; hence $S-\lambda$ is injective.
  To show surjectivity, set
  \[T=(S|{X_1}-\lambda)^{-1}A_1+(S|{X_2}-\lambda)^{-1}A_2.\]
  Then $(S-\lambda)T=A_1+A_2=S^{-2}$ from which we conclude that
  $\range(T)\subset\mD(S^3)$ and $(S-\lambda)S^2T=I$.
\end{proof}

The following example illustrates a typical situation in which
the projections from Lemma~\ref{lem:Apm} are unbounded.
It also shows that for fixed $S$ there may be several possible
choices for $A_1,A_2$.

\begin{example}\label{ex:unbproj}
  Let $S$ be the block diagonal operator on the sequence space $X=l^2(\C)$
  given by
  \[S=\diag(S_1,S_2,\dots), 
  \quad
  S_n=\pmat{n&2n^2\\0&-n}.\]
  First observe that
  \[(S_n-\lambda)^{-1}=
  \pmat{(n-\lambda)^{-1}&2n^2(n-\lambda)^{-1}(n+\lambda)^{-1}\\
    0&-(n+\lambda)^{-1}},\qquad \lambda\neq\pm n.\] 
  Then $\sup_n\|(S_n-\lambda)^{-1}\|<\infty$
  for $\lambda\not\in\Z\setminus\{0\}$. 
  Hence $\sigma(S)=\Z\setminus\{0\}$ and
  $(S-\lambda)^{-1}=\diag((S_1-\lambda)^{-1},(S_2-\lambda)^{-1},\dots)$,
  $\lambda\in\varrho(S)$.
  The spectral projections $P_{n}^{+}$, $P_{n}^{-}$ corresponding to the 
  eigenvalues $n$ and  $-n$ of $S_n$ are
  \[P_{n}^{+}=\pmat{1&n\\0&0},\quad P_{n}^{-}=\pmat{0&-n\\0&1}.\]
  Moreover 
  \[S_n^{-1}=\pmat{n^{-1}&2\\0&-n^{-1}},\quad
  S_n^{-2}=\pmat{n^{-2}&0\\0&n^{-2}}.\]
  Let $A_{n}^{\pm}:=S_n^{-2}P_{n}^{\pm}$. Then
  \[A_{n}^{+}=\pmat{n^{-2}&n^{-1}\\0&0},\quad
  A_{n}^{-}=\pmat{0&-n^{-1}\\0&n^{-2}}\]
  and we have 
  $A_{n}^{+}+A_{n}^{-}=S_n^{-2}$, $A_{n}^{+}A_{n}^{-}=A_{n}^{-}A_{n}^{+}=0$ and
  $A_{n}^{\pm}(S_n-\lambda)^{-1}=(S_n-\lambda)^{-1}A_{n}^{\pm}$
  for $\lambda\not\in\Z\setminus\{0\}$.

  Now we define an operator $A_1\in L(X)$ by choosing for each $n$ either
  $A_{n}^{+}$ or $A_{n}^{-}$; 
  for $A_2\in L(X)$ we take the complementary choice.
  More explicitly
  let $\Lambda_1\subset\N$ and $\Lambda_2 = \N \setminus \Lambda_1$.
  Set $\epsilon^j_n=+$\, if $n\in\Lambda_j$ and
  $\epsilon^j_n=-$\, if $n\not\in\Lambda_j$.
  Then the operators
  \begin{equation}\label{eq:Apm-example}
     A_j=\diag( A_{1}^{\epsilon^j_1}, A_{2}^{\epsilon^j_2}, A_{3}^{\epsilon^j_3},\dots ),
      \qquad j=1,2,
  \end{equation}
  satisfy all conditions in Lemma~\ref{lem:Apm}.
  The closed projections $P_j=S^2A_j$ are unbounded and block diagonal with 
  $P_j=\diag( P_{1}^{\epsilon^j_1}, P_{2}^{\epsilon^j_2}, P_{3}^{\epsilon^j_3},\dots )$.
  By Lemma~\ref{lem:Apm}, the subspaces $X_j=\range(P_j)$ are $S$-invariant.
  Clearly $X_j$ is the closed linear hull of the eigenvectors of $S$ for the eigenvalues $n\in\Lambda_j$ and the eigenvalues $-n$ for $n\not\in\Lambda_j$.
  Note that $X_1\oplus X_2 \neq X$.
  We will investigate this example further in Example~\ref{ex:unbsplit}.
\end{example}

\begin{remark}\label{rem:Bpm}
  Lemma~\ref{lem:Apm} continues to hold if $S^2$ and $S^{-2}$
  are replaced  throughout by $S$ and $S^{-1}$.
  That is, if there exist $B_1,B_2\in L(X)$ such that
  \begin{equation*}
    B_1+B_2=S^{-1}, \quad B_1B_2=B_2B_1=0,\quad
    B_jS^{-1}=S^{-1}B_j\quad (j=1,2),
  \end{equation*}
  then we obtain  closed 
  complementary projections $P_j=SB_j$ 
  where, in particular, $\mD(S)\subset\mD(P_j)$, $X_j=\range(P_j)$ are
  $S$- and $(S-\lambda)^{-1}$-invariant and \eqref{eq:Apm-spec} holds.
\end{remark}

\begin{remark}
  Finally, we can put Lemma~\ref{lem:Apm} in a bounded operator setting:
  Let $T\in L(X)$ be an injective operator
  whose inverse $T^{-1}:\range(T)\to X$ is possibly unbounded.
  If $B_j\in L(X)$ are such that
  \begin{equation*}
    B_1+B_2=T, \quad B_1B_2=B_2B_1=0,\quad
    B_jT=TB_j\quad (j=1,2),
  \end{equation*}
  then $P_j=T^{-1}B_j$ are closed complementary projections
  with $X_1 := \range(P_1) = \ker(B_2)$ and 
  $X_2 := \range(P_2) = \ker(B_1)$.
  The subspaces $X_j$ are $T$- and $(T-\lambda)^{-1}$-invariant and
  \begin{align}
    &\sigma(T)\setminus\{0\}=
    \big( \sigma(T|X_1)\setminus\{0\} \big)
    \cup
    \big( \sigma(T|X_2)\setminus\{0\} \big),\label{eq:Tpm-spec1}\\
    &\sigma(T)\supset\sigma(T|X_1) \cup \sigma(T|X_2)\label{eq:Tpm-spec2}
  \end{align}
  with equality in \eqref{eq:Tpm-spec2} if $T$ has dense  range.
  This follows from the previous remark with $S=T^{-1}$
  (so $0\in\varrho(S)$).
  To obtain \eqref{eq:Tpm-spec1}, note that
  for $\lambda\neq0$,
  $\lambda\in\sigma(S)$ if and only if
  $\lambda^{-1}\in\sigma(S^{-1})$,
  whereas \eqref{eq:Tpm-spec2}
  is a consequence of the $T^{-1}$-invariance of $X_j$.
  If $\range(T)$ is dense in $X$, then
  equality in \eqref{eq:Tpm-spec2}
  follows because $0\in\varrho(T|X_1)\cap\varrho(T|X_2)$ implies
  that $T^{-1}$ is bounded on $X_1\oplus X_2$,
  which is a dense subspace of $X$ since
  $\range(T)=\mD(S)\subset\mD(P_j)=X_1\oplus X_2$.
\end{remark}

\section{Dichotomous operators} 
\label{sec:dichot}

\begin{samepage}
\begin{definition}
   \label{def:dich}
   Let $X$ be a Banach space and let $X_\pm$ be closed subspaces with $X=X_+\oplus X_-$.
   An operator $S(X\to X)$ is called \emph{dichotomous} with respect to the decomposition $X=X_+\oplus X_-$ if
   \begin{enumerate}
      \item\label{enumi:dich:ir} $\I\R \subset \varrho(S)$,
      \item\label{enumi:dich:inv} $X_+$ and $X_-$ are $S$-invariant,
      \item 
      \label{enumi:dich:spec}
      $\sigma(S|X_+) \subset \C_+$, $\sigma(S|X_-) \subset \C_-$.
   \end{enumerate}

   An operator $S$ is called \emph{strictly dichotomous} with respect to
   $X=X_+\oplus X_-$ if, in addition,
   \begin{enumerate}
      \setcounter{enumi}{3}
      \item\label{enumi:dich:bnd}
	$\|(S|X_+-\lambda)^{-1}\|$ is bounded on $\overline{\C_-}$,\,
	$\|(S|X_--\lambda)^{-1}\|$ is bounded on $\overline{\C_+}$.
   \end{enumerate}
   A dichotomous operator
   $S$ is called \emph{exponentially dichotomous} if
   $-S|X_+$ and $S|X_-$ generate exponentially decaying
   semigroups.
\end{definition}
\end{samepage}

Clearly, exponential dichotomy implies strict dichotomy.
Note that all dichotomous operators appearing in \cite{LT2001,LRR2002,tretter-wyss} are in fact strictly dichotomous, see Remark~\ref{rem:theo-splitting}.

\begin{remark}
   For a given operator $S$, there may exist several decompositions of $X$ with respect to which it is dichotomous (Example~\ref{ex:nonunique}), but there exists at most one with respect to which it is strictly dichotomous (Lemma~\ref{lem:dichunique}).
\end{remark}

\begin{example}\label{ex:nonunique}
  Let $X$ be a Banach space and let $S(X\to X)$ be any operator satisfying $\sigma(S)=\varnothing$.
  Then evidently $S$ is dichotomous with respect to the two decompositions
  \begin{flalign}
     \label{eq:nonunique1}
     &\hspace*{20ex}& X_+ &= X,  & X_- &=\{0\} &\hspace*{20ex}&\\
     &\nonumber\text{and}\hidewidth\\
     \label{eq:nonunique2}
     &\hspace*{20ex}& X_+ &= \{0\},  & X_- &=X. &\hspace*{20ex}&
  \end{flalign}
  Examples for such operators are generators of nilpotent
  contraction semigroups, e.g.\ shift semigroups on bounded intervals.
  In this case, the resolvent $(S-\lambda)^{-1}$ is uniformly
  bounded on $\C_+$ and thus
  $S$ is strictly dichotomous with respect to the
  decomposition \eqref{eq:nonunique2} but not with respect to \eqref{eq:nonunique1}.

  To obtain an example of a dichotomous operator with non-empty spectrum, we consider
  an operator given as the direct sum $S=S_0\oplus S_1\oplus S_2$ on
  $X=X_0\oplus X_1\oplus X_2$ where the $S_j$ are linear operators on $X_j$
  such that  
  \[\sigma(S_0)=\varnothing,\quad
  \sigma(S_1)\subset \set{\lambda\in\C}{\re\lambda\geq h},\quad
  \sigma(S_2)\subset \set{\lambda\in\C}{\re\lambda\leq -h}\]
  for some $h>0$.
  Then $S$ is dichotomous with respect to the decompositions
  \begin{flalign}
     &\hspace*{15ex}& X_+ &= X_1\oplus X_0,  & X_- &= X_2 &\hspace*{15ex}&\\
     &\nonumber\text{and}\hidewidth\\
     &\hspace*{15ex}& X_+ &= X_1,  & X_- &=X_2\oplus X_0. &\hspace*{15ex}&
  \end{flalign}

\end{example}

\begin{remark}
   Even in the case of bisectorial operators, which will be introduced in Section~\ref{sec:bisect}, dichotomy is not sufficient to determine the subspaces $X_\pm$ uniquely, see Section~\ref{subsec:nonunique2}.
\end{remark}

\begin{samepage}
\begin{remark}
   \label{rem:genvec}
   Let $S$ be a dichotomous operator with respect to $X= X_+ \oplus X_-$.
   \begin{enumerate}
      \item If $x$ is a (generalised) eigenvector of $S$ with eigenvalue $\lambda\in \C_\pm$, then $x\in X_\pm$.

      \item Suppose that $S$ has a complete system of generalised eigenvectors.
      Then the spaces $X_\pm$ are uniquely determined by $S$ as
      the closures of the span of generalised eigenvectors of $S$ whose
      eigenvalues belong to $\C_\pm$.
   \end{enumerate}
\end{remark}
\end{samepage}

\begin{lemma}\label{lem:blockop}
  \begin{enumerate}
     \item\label{item:blockop:i}
  If $S$ is dichotomous with respect to $X=X_+\oplus X_-$, then
  $\mD(S)=(\mD(S)\cap X_+)\oplus(\mD(S)\cap X_-)$,
  and $S$ admits the block matrix representation
  \[S=\begin{pmatrix}S|X_+&0\\0&S|X_-\end{pmatrix}\]
  with respect to $X=X_+\oplus X_-$.
  In particular, $\sigma(S)=\sigma(S|X_+)\cup\sigma(S|X_-)$ and
  the subspaces $X_\pm$ are $(S-\lambda)^{-1}$-invariant for $\lambda\in\varrho(S)$.

  \item\label{item:blockop:ii}
  If $S$ is  strictly dichotomous, then there exist $h>0$ and $M>0$ such that
  $\set{\lambda\in\C}{|\re\lambda|\leq h}\subset\varrho(S)$ and
  \begin{equation}\label{eq:resolvbnd-h0}
     M=\sup_{|\re\lambda|\leq h}\|(S-\lambda)^{-1}\|<\infty.
  \end{equation}
  \end{enumerate}
\end{lemma}
\begin{proof}
   \enumiref{item:blockop:i}
   is proved as in \cite[Lemma~2.4]{tretter-wyss}.
   If $S$ is strictly dichotomous, then a Neumann series argument implies that there exist
   $h>0$ and $M>0$ such that $\lambda\in\varrho(S|X_\pm)$ 
   and $\|(S|X_\pm-\lambda)^{-1}\|\le M$ whenever $|\re\lambda|\leq h$.
   \enumiref{item:blockop:ii}
   then follows from the block matrix decomposition in \enumiref{item:blockop:i}.
\end{proof}

Now we establish the uniqueness of the decomposition $X=X_+\oplus X_-$ of a strictly dichotomous operator.
To this end, let $S(X\to X)$ with $\I\R\subset\varrho(S)$.
We consider the two subspaces $G_+$ and $G_-$ defined by
\begin{equation}\label{eq:G0}
  G_\pm=\set{x\in X}{(S-\lambda)^{-1}x\text{ has a bounded analytic 
    extension to }\overline{\C_\mp}}.
\end{equation}
More explicitly, for $x\in X$ we consider the analytic function
\begin{align*}
   \phi_x:\varrho(S)\to X,
   \qquad \phi_x(\lambda) = (S-\lambda)^{-1}x.
\end{align*}
Then $x\in G_+$ if and only if
$\phi_x$ admits a bounded analytic extension to the closed left half-plane
\begin{align*}
   \phi_x^+:\overline{\C_-}\cup\varrho(S)\to X.
\end{align*}
Analogously, $x\in G_-$ if and only if
$\phi_x$ admits a bounded analytic extension to the closed right half-plane
\begin{align*}
   \phi_x^-:\overline{\C_+}\cup\varrho(S)\to X.
\end{align*}

\begin{lemma}\label{lem:dichunique}
   Let $S(X\to X)$ with $\I\R\subset \varrho(S)$ and let $G_+$ and $G_-$ be given by \eqref{eq:G0}.
  Then:
  \begin{enumerate}
  \item\label{enumi:Gintersection}
  $G_+\cap G_-=\{0\}$.
  \item\label{enumi:dichunique}
  If $S$ is strictly dichotomous for the decomposition $X=X_+\oplus X_-$, then $X_\pm=G_\pm$.
  In particular, the subspaces $X_\pm$ are uniquely determined.
  \end{enumerate}
\end{lemma}
\begin{proof}
  \enumiref{enumi:Gintersection}:
  Let $x\in G_+\cap G_-$. Then $\phi_x(\lambda)=(S-\lambda)^{-1}x$
  can be extended to a bounded analytic function on $\C$
  which must be constant by Liouville's theorem.
  Hence $(S-\lambda_1)^{-1}x = (S-\lambda_2)^{-1}x$ for all $\lambda_1,\lambda_2\in\varrho(S)$ which is only the case if $\varrho(S)=\varnothing$ or $x=0$.

  \enumiref{enumi:dichunique}: For $x\in X_+$, strict dichotomy of $S$ implies that $\phi_x^+(\lambda)=(S|X_+-\lambda)^{-1}x$
  is a bounded analytic extension of $(S-\lambda)^{-1}x$ to $\overline{\C_-}$.
  Hence $x\in G_+$, i.e.\ $X_+\subset G_+$;
  similarly $X_-\subset G_-$.
  From $X=X_+\oplus X_-$ and $G_+\cap G_-=\{0\}$ we thus obtain
  $X_\pm=G_\pm$.
\end{proof}

\begin{remark}\label{rem:Gpm}
   The subspaces $G_\pm$ have been introduced in \cite{BGK1986},
   but with the roles of $G_+$ and $G_-$ interchanged and the additional conditions
   $\phi_x^\pm(\lambda)\in\mD(S)$ and
   $(S-\lambda)\phi_x^\pm(\lambda)=x$.
   In the  proof of Theorem~\ref{theo:splitting} we will show that
   if in addition to $\I\R\subset\varrho(S)$ an estimate
   \eqref{eq:resolvbnd-h0} holds,
   then $\phi_x^\pm(\lambda)=(S|G_\pm-\lambda)^{-1}x$ and hence both
   conditions are  automatically fulfilled.
\end{remark}

\begin{lemma}[Phragm\'en-Lindel\"of theorem for vector valued functions]
\label{theo:pl}%
Let $X$ be a Banach space,
  $a\geq1/2$ and 
  $\Sigma=\set{z\in\C}{|\arg z|<\frac{\pi}{2a}}$.
  Consider a continuous function $f:\overline{\Sigma}\to X$ which is
  analytic in $\Sigma\setminus\{0\}$.
  If there exist constants $M,C>0$ and $0<b<a$ such that
  $\|f(z)\|\leq M$ for $z\in\partial\Sigma$ and 
  $\|f(z)\|\leq C e^{|z|^b}$ for $z\in\Sigma$,
  then also $\|f(z)\|\leq M$ for $z\in\Sigma$.
\end{lemma}
\begin{proof}
   Let $x'\in X'$ with $\|x'\| = 1$.
   Then $x'\circ f$ is continuous in $\overline\Sigma$, analytic in $\Sigma$,
   $|(x'\circ f)(z)| \le \|f(z)\| \le M$ for $z\in\partial \Sigma$
   and $|(x'\circ f)(z)| \le Ce^{|z|^b}$ for $z\in\Sigma$.
   Hence, by the Phragm\'en-Lindel\"of theorem for complex valued functions, see e.g. \cite[Corollary~VI.4.2]{conway}, we obtain
   $|(x'\circ f)(z)| \le M$ for all $z\in\Sigma$.
   Since this is true for all $x'\in X'$ with $\|x'\|=1$, a corollary of the Hahn-Banach theorem implies that $\|f(z)\| \le M$ for all $z\in\Sigma$.
\end{proof}

The Phragm\'en-Lindel\"of theorem shows that
the condition on the boundedness of the extensions $\phi_x^\pm$ 
in the definition of $G_\pm$ can be weakened
if the operator $S$ satisfies an estimate \eqref{eq:resolvbnd-h0}.
\begin{proposition}
   Let $S(X\to X)$ and $h,M>0$ be such that
   $\set{\lambda\in\C}{|\re\lambda|\leq h}\subset\varrho(S)$ and
   \eqref{eq:resolvbnd-h0} holds.
   If $x\in X$ is such that $\phi_x(\lambda)=(S-\lambda)^{-1}x$ admits 
   an analytic extension $\phi_x^+$ to the left half-plane
   satisfying
   \[\|\phi_x^+(\lambda)\|\leq K|\lambda|^k,\quad \re\lambda<-h,\]
   with some constants $k\in\N$ and $K>0$, then 
   $\|\phi_x^+(\lambda)\|\le M$ for $\re\lambda\leq h$; in particular
   $x\in G_+$.
   Similarly if there is an analytic extension $\phi_x^-$ to the 
   right half-plane such that
   \[\|\phi_x^-(\lambda)\|\leq K|\lambda|^k,\quad \re\lambda>h,\]
   then 
   $\|\phi_x^-(\lambda)\|\le M$ for $\re\lambda\geq -h$;
   in particular $x\in G_-$.
\end{proposition}
\begin{proof}
   From \eqref{eq:resolvbnd-h0} we get
   $\|\phi_x^-(\lambda)\|\le M$ for $|\re\lambda| \leq h$.  
   Since by assumption $\phi_x^-$ is  polynomially bounded for $\re\lambda>h$,
   we can find for any $0<b<1$ a constant $C\geq M$ such that
   $\|\phi_x^-(\lambda)\|\leq Ce^{|z|^b}$ for $\re\lambda>0$.
   The function $\phi_x^-$ thus satisfies the hypothesis of Lemma~\ref{theo:pl}
   with $a=1$, $0<b<1$ arbitrary and $C\geq M$ large enough;
   hence for
   all $\lambda\in\Sigma=\C_+$ we obtain $\|\phi_x^-(\lambda)\| \leq M$.
   The proof for $\phi_x^+$ is similar with the Phragm\'en-Lindel\"of theorem
   applied to $\widetilde \phi_x^+(\lambda) :=\phi_x^+(-\lambda)$.
\end{proof}

\section{Spectral splitting along the imaginary axis} 
\label{sec:splitting}

In this section we prove our first spectral splitting result:
If the resolvent of $S$ is uniformly bounded along the imaginary axis,
then  the subspaces $G_\pm$ defined in \eqref{eq:G0}
are closed invariant subspaces of $S$ corresponding to its spectrum in
$\C_\pm$.
Moreover we construct a
pair of closed complementary projections $P_\pm$ onto $G_\pm$.

We make the following assumption: there exists $h>0$ such that
\begin{equation}\label{eq:resolvinc-h}
  \set{\lambda\in\C}{|\re\lambda|\leq h}\subset\varrho(S)
\end{equation}
 and
\begin{equation}\label{eq:resolvbnd-h}
  M:=\sup_{|\re\lambda|\leq h}\|(S-\lambda)^{-1}\|<\infty.
\end{equation}

For this assumption to hold it is sufficient that $\I\R\subset\varrho(S)$ and that the resolvent $(S-\lambda)^{-1}$ is uniformly bounded on $\I\R$.
If \eqref{eq:resolvinc-h} and \eqref{eq:resolvbnd-h}
hold for $h>0$, then they also hold for some $h'>h$ with a possibly larger $M$.
Both statements follow from a standard Neumann series argument.

Note that by Lemma~\ref{lem:blockop}\,\enumiref{item:blockop:ii} every strictly dichotomous operator satisfies \eqref{eq:resolvinc-h} and \eqref{eq:resolvbnd-h}.

\smallskip

For every $S$ which satisfies \eqref{eq:resolvinc-h} and \eqref{eq:resolvbnd-h} we can define the operators
\begin{equation}\label{eq:Apmdef}
   A_\pm=\frac{\pm 1}{2\pi\I}\int_{\pm h-\I\infty}^{\pm h+\I\infty}
   \frac{1}{\lambda^2}(S-\lambda)^{-1}\,d\lambda
\end{equation}
as in \cite{BGK1986}.
Note that by \eqref{eq:resolvbnd-h} the integrals converge in the uniform operator topology, hence $A_\pm$ are well-defined bounded linear operators.
Due to Cauchy's theorem, the integrals on the right hand side are independent of $h$ as long as \eqref{eq:resolvinc-h} and \eqref{eq:resolvbnd-h} hold.%
\smallskip

The next theorem extends the results from
\cite[Theorem~3.1]{BGK1986}.
In particular we obtain a spectral splitting also in the case of unbounded 
projections $P_\pm$.

\begin{theorem}\label{theo:splitting}
  Let $S(X\to X)$ be a linear operator on the Banach space $X$ satisfying
  \eqref{eq:resolvinc-h} and \eqref{eq:resolvbnd-h}. Then:
  \begin{enumerate}
  \item
    The subspaces $G_\pm$ defined in \eqref{eq:G0} are closed, $S$- and
    $(S-\lambda)^{-1}$-invariant and satisfy
    \begin{align}
      &\sigma(S)=\sigma(S|G_+)\cup\sigma(S|G_-),
      \label{eq:split:spec1}\\
      &\sigma(S|G_\pm)=\sigma(S)\cap\C_\pm,
      \label{eq:split:spec2}\\
      &\|(S|G_\pm-\lambda)^{-1}\|\leq M \quad\text{for}\quad
      \lambda\in\overline{\C_\mp},
      \label{eq:split:resolvbnd}
    \end{align}
    where $M$ is given by \eqref{eq:resolvbnd-h}.
  \item
  Let $A_\pm$ as in \eqref{eq:Apmdef}.
    Then the operators
    $P_\pm=S^2A_\pm$ are  closed complementary projections satisfying
    $\mD(P_+)=\mD(P_-)=G_+\oplus G_-$ and
    \[G_\pm=\range(P_\pm)=\ker(A_\mp).\]
    Hence $P_\pm$ are the complementary projections corresponding to the
    direct sum $G_+\oplus G_-\subset X$.
  \item\label{enumi:Ppm-repr}
    $\mD(S^2)\subset\mD(P_\pm)$ and
    \begin{equation}\label{eq:Ppmint}
      P_\pm x=\frac{\pm 1}{2\pi\I}\int_{\pm h-\I\infty}^{\pm h+\I\infty}
      \frac{1}{\lambda^2}(S-\lambda)^{-1}S^2x\,d\lambda,
      \quad x\in\mD(S^2).
    \end{equation}
  \end{enumerate}
\end{theorem}

\begin{proof}
  We first show (ii).
  In \cite{BGK1986} it has been shown that 
  $A_+A_-=A_-A_+=0$ and $A_++A_-=S^{-2}$. 
  For the convenience of the reader we sketch the proof.
  Using the resolvent identity and Fubini's theorem, we get
  \begin{align*}
     A_+A_-
     &= \int_{h-\I\infty}^{h+\I\infty}
    \int_{-h-\I\infty}^{-h+\I\infty}
    \frac{1}{\lambda^2\mu^2} (S-\lambda)^{-1} (S-\mu)^{-1}
    \,\rd\mu\,\rd\lambda
    \\[1ex]
    &=
    \int_{h-\I\infty}^{h+\I\infty}
    \biggl(\int_{-h-\I\infty}^{-h+\I\infty}
    \frac{d\mu}{\mu^2(\lambda-\mu)}\biggr)
    \frac{1}{\lambda^2}(S-\lambda)^{-1} \,\rd\lambda
    \\
    & \phantom{=\ } 
    -
    \int_{-h-\I\infty}^{-h+\I\infty}
    \biggl(\int_{h-\I\infty}^{h+\I\infty}
    \frac{d\lambda}{\lambda^2(\lambda-\mu)}\biggr)
    \frac{1}{\mu^2}(S-\mu)^{-1} 
    \,\rd\mu
  \end{align*}
  and hence $A_+A_-=0$ since the integrals in parentheses vanish.
  Clearly $A_\pm$ and $(S-\lambda)^{-1}$ commute,
  so $A_+$ and $A_-$ commute too and $A_-A_+=0$.
  By Cauchy's theorem we obtain
  \[A_++A_-=\frac{1}{2\pi\I}\int_\Gamma\frac{1}{\lambda^2}(S-\lambda)^{-1}
  \,d\lambda=S^{-2}\]
  where $\Gamma$ is a small positively oriented circle around the origin.
  Therefore, by Lemma~\ref{lem:Apm}, $P_\pm=S^2A_\pm$ are closed complementary projections with $\range(P_\pm)=\ker(A_\mp)$.

  We show next that $G_\pm=\ker(A_\mp)$.
  Let $x\in G_+$. Then $(S-\lambda)^{-1}x$ has a bounded analytic extension
  $\phi_x^+$ to $\C_-$. Consequently
  \[A_-x=\frac{-1}{2\pi\I}\int_{-h-\I\infty}^{-h+\I\infty}\frac{1}{\lambda^2}
  (S-\lambda)^{-1}x\,d\lambda=0\]
  by Cauchy's theorem.
  We thus have $G_+\subset\ker(A_-)$.
  To show the converse inclusion,
  let $\re z<-h$ and consider the bounded operator
  \begin{equation}\label{eq:Rmz}
    R_-(z)=\frac{1}{2\pi\I}\int_{-h-\I\infty}^{-h+\I\infty}
    \frac{z^2}{\lambda^2(\lambda-z)}(S-\lambda)^{-1}\,d\lambda.
  \end{equation}
  From $(S-z)(S-\lambda)^{-1}=I+(\lambda-z)(S-\lambda)^{-1}$
  and the closedness of $S$ we obtain $\range(R_-(z))\subset\mD(S)$
  and
  \[(S-z)R_-(z)=\frac{1}{2\pi\I}\int_{-h-\I\infty}^{-h+\I\infty}
  \frac{z^2}{\lambda^2(\lambda-z)}\,d\lambda -z^2A_-
  =I-z^2 A_-;\]
  thus $(S-z)R_-(z)=I$ on $\ker(A_-)$.
  By Lemma~\ref{lem:Apm},
  $\ker(A_-)$ is $S$- and $(S-\lambda)^{-1}$-invariant, and hence also $R_-(z)$-invariant.
  Since $(S-z)R_-(z)=R_-(z)(S-z)$ on $\mD(S)$, we conclude that
  $\C_-\subset \varrho(S|\ker(A_-))$ and for all $x\in\ker(A_-)$
  \begin{equation}\label{eq:resolv-G+}
    (S|\ker(A_-)-z)^{-1}x
    =\begin{cases}
    (S-z)^{-1}x,\quad &|\re z|\leq h,\\
    R_-(z)x,\quad &\re z<-h.\end{cases}
  \end{equation}
  The definition of $R_-(z)$ in conjunction with \eqref{eq:resolvbnd-h}
  implies that $\|R_-(z)\|\leq C_0|z|^2$ for 
  $\re z<-h$ with some constant $C_0$; as remarked earlier, 
  we may replace $h$ by $h/2$ in \eqref{eq:Rmz}.
  Together with \eqref{eq:resolvbnd-h}
  we see that the function 
  $f(z)=(S|\ker(A_-)-z)^{-1}$ with values in the Banach space $L(X)$
  is analytic and polynomially bounded for $\re z<0$
  and bounded by $M$ on $\I\R$.
  Hence
  the Phragm\'en-Lindel\"of theorem (Lemma~\ref{theo:pl}
  applied to $f(-z)$  with $a=1$, $0<b<1$ arbitrary and $C$ large enough)
  yields
  \begin{equation}\label{eq:resolvbnd-A-}
    \|(S|\ker(A_-)-z)^{-1}\|\leq M, \quad z\in\overline{\C_-}.
  \end{equation}
  As a consequence, if $x\in\ker(A_-)$, then 
  $(S|\ker(A_-)-z)^{-1}x$
  is a bounded analytic extension of $(S-z)^{-1}x$
  to $\overline{\C_-}$
  and hence $x\in G_+$.
  We have thus shown that $G_+=\ker(A_-)$ and the proof of $G_-=\ker(A_+)$
  is analogous. 
  In particular, we proved
  $\sigma(S|G_\pm)\subset\C_\pm$,
  \eqref{eq:split:resolvbnd} and
  \begin{align*}
     \mD(P_+)=\mD(P_-)=\range(P_+)\oplus\range(P_-)=G_+\oplus G_-.
  \end{align*}
  All remaining statements in (i) and (iii) follow from
  Lemma~\ref{lem:Apm}.
\end{proof}

\begin{corollary}\label{coroll:dichot0}
  Suppose $S(X\to X)$ satisfies \eqref{eq:resolvinc-h} and
  \eqref{eq:resolvbnd-h}. 
  Then the following are equivalent:
  \begin{enumerate}
     \item $S$ is strictly dichotomous;
     \item $X=G_+\oplus G_-$;
     \item $P_+\in L(X)$.
  \end{enumerate}
  In this case $X=G_+\oplus G_-$ is the corresponding spectral decomposition.
\end{corollary}
\begin{proof}
  In Lemma~\ref{lem:dichunique} we have already seen that strict dichotomy implies
  $X_\pm=G_\pm$ and hence $X=G_+\oplus G_-$.
  Conversely, if $X=G_+\oplus G_-$, then Theorem~\ref{theo:splitting} implies that $S$ is strictly dichotomous for the choice $X_\pm = G_\pm$.
  Finally (ii)$\Leftrightarrow$(iii) follows from
  $\mD(P_+)=G_+\oplus G_-$
  and the closed graph theorem.
\end{proof}

In Theorem~\ref{theo:splitting} we did not assume that $S$ is densely defined. 
If it is, then the projections $P_\pm$
are densely defined too, and 
we obtain a nice  criterion for strict dichotomy.

{\samepage
\begin{corollary}\label{coroll:dichot}
  Let $S(X\to X)$ be densely defined and satisfy
  \eqref{eq:resolvinc-h} and \eqref{eq:resolvbnd-h}. Then
  $G_+\oplus G_-\subset X$ is dense and
  the closed projections $P_\pm$ are densely defined.
  Moreover, the following assertions are equivalent:
  \begin{enumerate}
     \item $S$ is strictly dichotomous;
     \item $P_+$ is bounded on some dense subspace of $\mD(P_+)$;
     \item the operator $P$ defined by
     \begin{equation}\label{eq:Pint}
	Px=\frac{1}{2\pi\I}\int_{h-\I\infty}^{h+\I\infty}\frac{1}{\lambda^2}
	(S-\lambda)^{-1}S^2x\,d\lambda,\quad x\in\mD(S^2),
     \end{equation}
     is bounded on some dense subspace of $\mD(S^2)$.
  \end{enumerate}
  In this case $P_+$ is the unique bounded extension of $P$ to $X$.
\end{corollary}
}

\begin{proof}
  The first assertions are immediate from
  $\mD(S^2)\subset\mD(P_\pm)=G_+\oplus G_-$.

  (i)$\Rightarrow$(ii): This follows from Corollary~\ref{coroll:dichot0}.

  (ii)$\Rightarrow$(i): If $P_+$ is bounded on a dense subspace,
  then the closedness of $P_+$ implies that 
  $\mD(P_+)=X$. Thus $P_+\in L(X)$ and $S$ is strictly dichotomous.

  (ii)$\Leftrightarrow$(iii) and the final assertion: 
  This is clear since \eqref{eq:Ppmint} implies that $P$ is a restriction of 
  $P_+$.
\end{proof}
  
\begin{remark}\label{rem:theo-splitting}
  Theorem~\ref{theo:splitting} and Corollaries~\ref{coroll:dichot0}
  and~\ref{coroll:dichot} generalise and extend the results
  from \cite[Theorem~3.1]{BGK1986}. 

  \begin{enumerate}
  \item 
    Probably the main result of \cite[Theorem~3.1]{BGK1986} 
    on spectral splitting for densely defined $S$ is that if
    the operator $P$ in \eqref{eq:Pint} is bounded 
    on $\mD(S^2)$, then $S$ is dichotomous.
    Corollary~\ref{coroll:dichot} shows that the boundedness on some dense subspace of $\mD(S^2)$ is sufficient and that $S$ is even strictly dichotomous in this case.
    Note that since
    \cite{LT2001,LRR2002,tretter-wyss} all use
    \cite[Theorem~3.1]{BGK1986} to prove dichotomy, the corresponding
    operators in those papers are in fact strictly dichotomous.

  \item
    We have shown that also in
    the non-dichotomous case, i.e.,\ when $P_\pm$ are unbounded,
    $G_\pm$ are closed invariant subspaces with
    $\sigma(S|G_\pm)\subset \C_\pm$ and
    $\sigma(S|G_+)\cup\sigma(S|G_-)= \sigma(S)$.
    In \cite{BGK1986} the closedness and $S$-invariance of $G_\pm$
    has only been obtained for exponentially dichotomous $S$.
  \item
    Using the Phragm\'en-Lindel\"of theorem, we showed $G_\pm=\ker(A_\mp)$
    where\-as in \cite{BGK1986} only the inclusion $G_\pm\subset\ker(A_\mp)$ 
    was obtained.
  \item 
    We showed that $P_\pm$ as defined in Theorem~\ref{theo:splitting}\,(ii) are closed projections even if they are unbounded.
    This is used e.g.\  in 
    Corollary~\ref{coroll:dichot} where it is sufficient to check 
    boundedness on any dense subspace.
    In Section~\ref{sec:pert} this allows us to prove perturbation results under weaker conditions on the domains of the involved operators.
  \item
    In \cite{BGK1986} it is always assumed that $S$ is densely defined.
    Theorem~\ref{theo:splitting} and Corollary~\ref{coroll:dichot0} are
    valid also if $S$ is not densely defined.
  \end{enumerate}
\end{remark}

\begin{example}\label{ex:unbsplit}
   We continue Example~\ref{ex:unbproj}.
   A straightforward calculation shows that 
   $\sup_{\lambda\in\I\R}\|(S-\lambda)^{-1}\|\leq 3$
   and hence \eqref{eq:resolvinc-h} and \eqref{eq:resolvbnd-h}
   are satisfied.
   From Theorem~\ref{theo:splitting}, in particular \eqref{eq:Apmdef},
   it follows that $A_\pm$ and $P_\pm$ are block diagonal
   and given by
   \begin{gather*}
      A_+=\diag(A_{1}^{+},A_{2}^{+},\dots),\quad 
      A_-=\diag(A_{1}^{-},A_{2}^{-},\dots),\\
      P_+=\diag(P_{1}^{+},P_{2}^{+},\dots),\quad 
      P_-=\diag(P_{1}^{-},P_{2}^{-},\dots),
   \end{gather*}
   where $A_{n}^{\pm}$ and $P_{n}^{\pm}$ are as in Example~\ref{ex:unbproj}. 
   The operators $P_+$ and $P_-$ are unbounded
   and $S$ is not dichotomous, compare Remark~\ref{rem:genvec}.
\end{example}

\begin{remark}
   Since $P_\pm$ are closed operators with $\mD(S^2)\subset \mD(P_\pm)$, their restrictions $P_\pm|\mD(S^2)$ are bounded in the graph norm of $S^2$.
   In the almost bisectorial case, we even have $\mD(S)\subset \mD(P_\pm)$, and therefore $P_\pm|\mD(S)$ are bounded in the graph norm of $S$ (see Theorem~\ref{theo:almbisect}).
   In the general case it is not clear if $\mD(S)\subset \mD(P_\pm)$ and if the restriction $P_\pm|\mD(S)$ is bounded.
\end{remark}

\section{Bisectorial and almost bisectorial operators} 
\label{sec:bisect}

In this section we investigate the spectral splitting of bisectorial and
almost bisectorial operators.  Their resolvent norm is not only bounded on the
imaginary axis as assumed in Section~\ref{sec:splitting}, it even decays like
$|\lambda|^{-\beta}$.
\smallskip

For $0\le \theta\le \pi$ we define the sectors
\begin{equation}
   \label{eq:secdef}
   \Sigma_\theta = 
   \{ z\in\C : |\arg z | \le \theta \}.
\end{equation}
Let us first recall the notion of sectorial
and almost sectorial operators,
see e.g.\ \cite{haase,periago-straub}.

\begin{definition}
  \label{def:sec}
  A linear operator $S(X\to X)$ is called \emph{sectorial}
  if there exist
  $0<\theta<\pi$ and $M>0$ such that
  $\sigma(S)\subset\Sigma_\theta$ and
  \begin{equation}\label{eq:sect}
    \|(S-\lambda)^{-1}\|\leq\frac{M}{|\lambda|}
    \quad\text{for}\quad \lambda\in\C\setminus\Sigma_\theta.
  \end{equation}
  $S$ is called \emph{almost sectorial}
  if there exist $0<\beta<1$, $0<\theta<\pi$ and $M>0$ such that
  $\sigma(S)\subset\Sigma_\theta$ and
  \begin{equation}\label{eq:almsect}
    \|(S-\lambda)^{-1}\|\leq\frac{M}{|\lambda|^\beta}
    \quad\text{for}\quad \lambda\in\C\setminus\Sigma_\theta.
  \end{equation}
\end{definition}

\begin{remark}
  There are subtle differences in the behaviour of 
  sectorial and almost sectorial operators:
  \begin{enumerate}
  \item
    If $S$ is sectorial with angle $\theta$, then \eqref{eq:sect} also holds 
    with a smaller angle
    $0<\theta'<\theta$, though the constant $M$ 
    may be bigger for $\theta'$ as can be shown by a simple Neumann series argument.

  \item
    If $S$ is almost sectorial, then $0\in\varrho(S)$ (see e.g.\ \cite[Remark~2.2]{periago-straub}).

  \item
    If $S$ is sectorial and $0\in\varrho(S)$, then $S$ is almost sectorial.
    This is true because $\|(S-\lambda)^{-1}\|$ is bounded in a neighbourhood
    of zero, and we easily obtain \eqref{eq:almsect} from \eqref{eq:sect} (with
    a different constant $M$).
  \end{enumerate}
  Note that (i) is not necessarily true for almost sectorial
  operators and that sectorial operators may have zero in their spectrum.

\end{remark}

If we require the resolvent estimates only on the imaginary
axis and allow spectrum on both sides of it, we obtain
the definition of bisectorial and almost bisectorial operators.

\begin{definition}
   \label{def:bisec}
  An operator $S(X\to X)$ is called \emph{bisectorial} if 
  $\I\R\setminus\{0\}\subset\varrho(S)$ and
  \begin{equation}\label{eq:bisect}
    \|(S-\lambda)^{-1}\|\leq \frac{M}{|\lambda|},\quad 
    \lambda\in\I\R\setminus\{0\}.
  \end{equation}
  $S$ is called \emph{almost bisectorial} if 
  $\I\R\setminus\{0\}\subset\varrho(S)$ and
  there exists $0<\beta<1$ such that
  \begin{equation}\label{eq:almbisect}
    \|(S-\lambda)^{-1}\|\leq \frac{M}{|\lambda|^\beta},\quad 
    \lambda\in\I\R\setminus\{0\}.
  \end{equation}
\end{definition}
Bisectorial operators have been studied e.g.\ in
\cite{arendt-zamboni,tretter-wyss}.
\smallskip

The following results are analogous to the sectorial case.
They imply that almost bisectorial operators and bisectorial operators with
$0\in\varrho(S)$ fulfil the assumptions of Theorem~\ref{theo:splitting}. 
\begin{remark}\label{rem:bisect}
  \begin{enumerate}
  \item\label{enumi:bisect}
    If $S$ is bisectorial, then there exists $0<\theta<\pi/2$ such that
    the bisector
    \[\Omega_\theta=
    \C\setminus \big( \Sigma_\theta \cup (-\Sigma_\theta)\big)
    =
    \set{\lambda\in\C}{\theta < |\arg\lambda| < \pi-\theta}\]
    belongs to $\varrho(S)$ and \eqref{eq:bisect} holds on $\Omega_\theta$,
    see Figure~\ref{fig:omega}.
  \item\label{enumi:almbisect-resolv}
    If $S$ is almost bisectorial, then $0\in\varrho(S)$.
  \item\label{enumi:bisect-impl}
    If $S$ is bisectorial and $0\in\varrho(S)$, then $S$ is almost
    bisectorial.

  \item If $S$ is almost bisectorial or bisectorial with $0\in\varrho(S)$, then $S$ satisfies \eqref{eq:resolvinc-h} and \eqref{eq:resolvbnd-h}
    from Section~\ref{sec:splitting}.
  \end{enumerate}
\end{remark}

Similar to Remark~\ref{rem:bisect}\,\eqref{enumi:bisect},
the resolvent estimate of an almost bisectorial operator
actually holds inside a whole region around $\I\R$:

\begin{figure}

   \begin{tikzpicture}[scale=0.9]
      \begin{scope}
	 \clip (-2.8, -2.1) rectangle (2.8,2.1);

	 \draw[fill=yellow](60:3)--(60:-3)--(-60:3)--(-60:-3);
	 \draw[dashed,fill=orange!70!yellow](40:-4)--(40:4)--(-40:4)--(-40:-4);

	 \draw[->] (1,0) arc(0:60:1);
	 
	 \node at (48:1.3){$\theta$};
	 \node at (-.4,1.5){$\Omega_\theta$};
	 \node at (1.8,.5){$\sigma(S)$};
      \end{scope}

	 \draw[->] (-3,0) -- (3,0) node[right] {};
	 \draw[->] (0,-2.4) -- (0,2.4) node[above] {};
   \end{tikzpicture}
   \hspace*{\fill}
   \begin{tikzpicture}[scale=0.9]
      \begin{scope}
	 \clip (-2.8, -2.1) rectangle (2.8,2.1);
	 \draw[domain=-2:2,smooth,variable=\x,black,fill=yellow] plot ({\x},{1.5*\x*\x});
	 \draw[domain=-2:2,smooth,variable=\x,black,fill=yellow] plot ({\x},{-1.5*\x*\x});

	 \draw[black,dashed,fill=orange!70!yellow]
	 plot[domain=0:-3,smooth,variable=\x,black,dashed] ({\x},{-.6*\x*\x})
	 -- plot[domain=-3:0,smooth,variable=\x,black,dashed] ({\x},{.6*\x*\x});

	 \draw[black,dashed,fill=orange!70!yellow]
	 plot[domain=0:3,smooth,variable=\x,black,dashed] ({\x},{-.6*\x*\x})
	 -- plot[domain=3:0,smooth,variable=\x,black,dashed] ({\x},{.6*\x*\x});
      \end{scope} 

      \node at (-.5,1.5) {$\Omega$};
      \node at (2,.8){$\sigma(S)$};

      \draw[->] (-3,0) -- (3,0) node[right] {};
      \draw[->] (0,-2.4) -- (0,2.4) node[above] {};
   \end{tikzpicture}

   \caption{\label{fig:omega}Location of the spectrum of a bisectorial and an almost bisectorial operator.}
\end{figure}
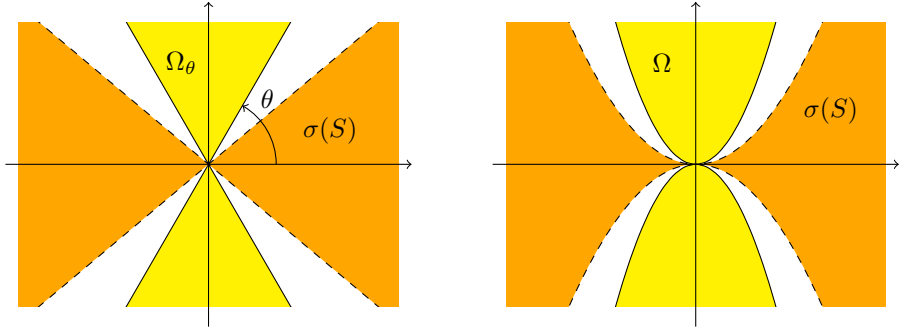%

\begin{lemma}\label{lem:almbisect}
  Let $S(X\to X)$ be almost bisectorial with constants
  $0<\beta<1$ and $M>0$ as in \eqref{eq:almbisect}.
  Then for every $\alpha<1/M$ the parabola shaped region
  \begin{equation*}
    \Omega=
    \set{a+\I b\in\C\setminus\{0\}}{|a|\leq \alpha|b|^\beta}
  \end{equation*}
  belongs to the resolvent set, $\Omega\subset\varrho(S)$,
  and  \eqref{eq:almbisect}  holds for all
  $\lambda\in\Omega$ (typically with a larger constant $M$),
  see Figure~\ref{fig:omega}.
\end{lemma}
\begin{proof}
  Let $\lambda=a+\I b\in\Omega$. Then the identity
  \[S-\lambda=\left(I-a(S-\I b)^{-1}\right)(S-\I b)\]
  and the estimate
  \[\|a(S-\I b)^{-1}\|\leq\frac{M|a|}{|b|^\beta}\leq \alpha M<1\]
  imply $\lambda\in\varrho(S)$ and thus $\Omega\subset\varrho(S)$.
  Moreover
  \[\|(S-\lambda)^{-1}\|\leq\frac{1}{1-\|a(S-\I b)^{-1}\|}\|(S-\I b)^{-1}\|
  \leq\frac{M}{(1-\alpha M)|b|^\beta}.\]
  Now if also $|b|\geq1$, then
  $|\lambda|\leq \alpha |b|^\beta+|b|\leq(1+\alpha)|b|$ and hence
  \[\|(S-\lambda)^{-1}\|\leq\frac{M(1+\alpha)^\beta}{(1-\alpha M)|\lambda|^\beta},
  \quad
  \lambda=a+\I b\in\Omega,\,|b|\geq1.\]
  Since $0\in\varrho(S)$ by Remark~\ref{rem:bisect}\eqref{enumi:almbisect-resolv}
  and $(S-\lambda)^{-1}$ is uniformly bounded on compact subsets of
  $\varrho(S)$, the proof is complete.
\end{proof}

A similar result can be shown for an almost sectorial operator $S$.
Let $\theta$ as in Definition~\ref{def:sec}.
Then $\varrho(S)$ contains a parabola around every ray 
$\{r\e^{\I\phi} : r\ge 0\}$ with $\theta\le |\phi| \le \pi$.
\medskip

In the rest of this section we will investigate the spectral splitting
properties of bisectorial and almost bisectorial operators.
Compared with Theorem~\ref{theo:splitting},
we obtain simplified formulas for
the projections $P_\pm=S^2A_\pm$ and show that the restrictions $S|G_\pm$ to
the spectral subspaces are sectorial and almost sectorial, respectively.

For an almost bisectorial operator $S$ let
\begin{equation}\label{eq:Bpmdef}
   B_\pm=\frac{\pm 1}{2\pi\I}\int_{\pm h-\I\infty}^{\pm h+\I\infty}
   \frac{1}{\lambda}(S-\lambda)^{-1}\,d\lambda
\end{equation}
with $h>0$ small enough.
By \eqref{eq:almbisect} and Lemma~\ref{lem:almbisect} the integrals converge
in the uniform operator topology.
Hence $B_\pm$ are well-defined bounded
linear operators and, due to Cauchy's theorem, the integrals on the right hand
side are independent of $h$ for $h$ small enough.

\begin{theorem}\label{theo:almbisect}
   Let $S(X\to X)$ be almost bisectorial and $P_\pm$ as in Theorem~\ref{theo:splitting}.
   Then:
   \begin{enumerate}
      \item\label{enumi:PB}
      $P_\pm=SB_\pm$,
      $\mD(S)\subset\mD(P_\pm)$ and
      \begin{equation}\label{eq:Ppm-int-defS}
	 P_\pm x=\frac{\pm 1}{2\pi\I}\int_{\pm h-\I\infty}^{\pm h+\I\infty}
	 \frac{1}{\lambda}(S-\lambda)^{-1}Sx\,d\lambda,
	 \quad x\in\mD(S).
      \end{equation}
      \item\label{enumi:restr-almsect}
      $\pm S|G_\pm$ are almost sectorial with angle $\theta=\pi/2$
      and unchanged constants $M,\beta$.
      \item\label{enumi:restr-sect}
      Let $S$ be bisectorial with $0\in\varrho(S)$ and $\theta$ as in Remark~\ref{rem:bisect}\,\enumiref{enumi:bisect}.
      Then $\pm S|G_\pm$ are sectorial with angle $\theta$
      and unchanged constant $M$.
   \end{enumerate}
\end{theorem}

\begin{proof}
  \enumiref{enumi:PB}:
  Using the resolvent identity
  $(S-\lambda)^{-1}=S^{-1}+\lambda(S-\lambda)^{-1}S^{-1}$
  we obtain from \eqref{eq:Apmdef}
  \[A_\pm=\frac{\pm 1}{2\pi\I}\int_{\pm h-\I\infty}^{\pm h+\I\infty}
  \frac{1}{\lambda^2}S^{-1}\,d\lambda
  +\frac{\pm 1}{2\pi\I}\int_{\pm h-\I\infty}^{\pm h+\I\infty}
  \frac{1}{\lambda}(S-\lambda)^{-1}S^{-1}\,d\lambda=B_\pm S^{-1}\]
  because the first integral vanishes by Cauchy's theorem.
  Moreover $S^{-1}B_\pm=B_\pm S^{-1}$.
  Therefore $P_\pm=S^2A_\pm=S^2S^{-1}B_\pm=SB_\pm$.
  For $x\in\mD(S)$ we get $B_\pm Sx=SB_\pm x$ and in particular 
  $x\in\mD(P_\pm)$.

  \enumiref{enumi:restr-almsect}:
  Consider $\lambda^\beta=\exp(\beta\log\lambda)$ 
  where $\log$ is a branch
  of the logarithm on $\C_-$. Then the mapping $\lambda\mapsto\lambda^\beta$
  is analytic on $\C_-$
  and continuous on $\overline{\C_-}$.
  The almost bisectoriality of $S$ yields
  $\|\lambda^\beta(S|G_+-\lambda)^{-1}\|\leq M$ for $\lambda\in\I\R$ since 
  $(S|G_+-\lambda)^{-1}=(S-\lambda)^{-1}|G_+$.
  Here we used that $0\in\varrho(S)$ by 
  Remark~\ref{rem:bisect}\,\enumiref{enumi:almbisect-resolv}.
  The $L(X)$-valued function $f(\lambda)=\lambda^\beta(S|G_+-\lambda)^{-1}$
  is thus analytic on $\C_-$ and bounded by $M$ on $\I\R$.
  Since $(S|G_+-\lambda)^{-1}$ is bounded on $\C_-$
  by Theorem~\ref{theo:splitting}, $f$ is polynomially bounded
  on $\C_-$.
  The Phragm\'en-Lindel\"of theorem (Lemma~\ref{theo:pl} with $a=1$ applied 
  to $f(-\lambda)$) then yields
  $\|\lambda^\beta(S|G_+-\lambda)^{-1}\|\leq M$ for 
  $\lambda\in\overline{\C_-}$.
  The proof for $S|G_-$ is analogous. 
  With $\beta=1$ we obtain \enumiref{enumi:restr-sect}.
\end{proof}

Statements \eqref{enumi:restr-almsect} and \eqref{enumi:restr-sect} of the Theorem
remain true when $S|G_\pm$ are replaced by
the operators $S_{M\pm}$ 
from Section~\ref{sec:Mpm} as follows from Lemma~\ref{lem:Mpm}\,\enumiref{enumi:Spm-resolv}.

\begin{remark}
  The operators $B_\pm$ satisfy the relations
  \[B_++B_-=S^{-1},\qquad B_+B_-=B_-B_+=0.\]
  These identities can be obtained either
  from the corresponding relations for $A_\pm$
  via  $A_\pm=S^{-1}B_\pm=B_\pm S^{-1}$ 
  or  from \eqref{eq:Bpmdef} by direct computation.
  The latter approach, together with Remark~\ref{rem:Bpm},
  yields an alternative proof of Theorem~\ref{theo:splitting}
  for  almost bisectorial operators.
  In the bisectorial case, this was used in \cite{arendt-zamboni}.
  There the sectoriality of the spectral parts $S|G_\pm$ was obtained
  (\cite[p.~215]{arendt-zamboni}), but not the $S$-invariance of $G_\pm$
  and the decomposition \eqref{eq:split:spec1} of the spectrum.
\end{remark}

To illustrate the situation of Theorem~\ref{theo:almbisect},
we consider
an almost bisectorial operator which is not dichotomous and whose
projections $P_\pm$ are (therefore) unbounded.
This is a variant of Example~\ref{ex:unbproj} and~\ref{ex:unbsplit}.
\begin{example}\label{ex:almbisect}
  Let $0<p<1$ and consider the block diagonal operator $S$
  on $X=l^2$
  given by
  \[S=\diag(S_1,S_2,\dots), 
  \quad
  S_n=\pmat{n&2n^{1+p}\\0&-n}.\]
  Direct calculations similar to Example~\ref{ex:unbproj}
  yield $\lim_{n\to\infty}\|(S_n-\lambda)^{-1}\|=0$ whenever
  $\lambda\not\in\Z\setminus\{0\}$.
  Hence $\sigma(S)=\Z\setminus\{0\}$ and $S$ has a compact resolvent.
  Moreover,
  \[\|(S-\lambda)^{-1}\|\leq\frac{M}{|\lambda|^{1-p}},
  \quad \lambda\in\I\R\setminus\{0\},\]
  i.e., $S$ is almost bisectorial.
  The spectral projections for $S_n$ corresponding to the eigenvalues $n$
  and $-n$, respectively, are
  \[P_{n}^{+}=\pmat{1&n^p\\0&0},\qquad P_{n}^{-}=\pmat{0&-n^p\\0&1}.\]
  Consequently, 
  $P_\pm=\diag(P_1^\pm,P_2^\pm,\dots)$
  are unbounded and
  $S$ is not dichotomous, compare Remark~\ref{rem:genvec}.
\end{example}%
If in the above example we choose $p=0$, then $S$ becomes bisectorial
\emph{and} strictly dichotomous.
In general however, bisectoriality (with $0\in\varrho(S)$)
does not imply dichotomy.
A counterexample was given in \cite{mcintosh-yagi}
(see Example~\ref{ex:mcintosh-yagi}).

The identity \eqref{eq:Ppm-int-defS} for the projections $P_\pm$ 
from Theorem~\ref{theo:almbisect} can be rearranged to yield another 
integral representation.
In the bisectorial case, this representation has
been obtained and used extensively in \cite{LT2001,tretter-wyss}.

\begin{corollary}\label{cor:iR-int}
  Let $S(X\to X)$ be almost bisectorial.
  Then:
  \begin{enumerate}
  \item\label{enumi:iR-int-eq}
    For every $x\in\mD(S)$,
    \begin{equation}\label{eq:iR-int}
      2P_+x-x=P_+x-P_-x=\frac{1}{\pi\I}
      \int_{-\I\infty}^{\I\infty\,\prime} (S-\lambda)^{-1}x\,d\lambda;
    \end{equation}
    in particular, the integral exists for all $x\in\mD(S)$.
    Here the prime denotes the Cauchy principal value at infinity.
  \item
    If, in addition,
    $S$ is densely defined and there exists a dense subspace
    $D\subset\mD(S)$ such that
    \[\int_{-\I\infty}^{\I\infty\,\prime}
    (S-\lambda)^{-1}x\,d\lambda, \qquad x\in D,\]
    defines a bounded operator, then
    $P_+$ is bounded and hence $S$ is strictly dichotomous.
  \end{enumerate}
\end{corollary}
\begin{proof}
  (i) 
  From \eqref{eq:Ppm-int-defS} we get for $x\in\mD(S)$,
  \begin{align*}
    P_+x&=\frac{1}{2\pi\I}\int_{h-\I\infty}^{h+\I\infty}
    \frac{1}{\lambda}(S-\lambda)^{-1}Sx\,d\lambda
    =\frac{1}{2\pi\I}\int_{h-\I\infty}^{h+\I\infty}
    \biggl(\frac{1}{\lambda}x+(S-\lambda)^{-1}x\biggr)\,d\lambda\\
    &=\frac{1}{2}x+
    \frac{1}{2\pi\I}\int_{h-\I\infty}^{h+\I\infty\,\prime}
    (S-\lambda)^{-1}x\,d\lambda
    =\frac{1}{2}x+
    \frac{1}{2\pi\I}\int_{-\I\infty}^{\I\infty\,\prime}
    (S-\lambda)^{-1}x\,d\lambda.
  \end{align*}
  Note that in the last step we used Cauchy's integral theorem and 
  \eqref{eq:almbisect}. The assertion follows from $x=P_+x+P_-x$ for 
  $x\in\mD(P_+)$.

  (ii)
  The assumption together with (i) implies that $P_+$ is bounded on
  the dense subspace $D$. 
  Corollary~\ref{coroll:dichot} yields the claim.
\end{proof}

\begin{remark}
  In \cite{LT2001} the representation \eqref{eq:iR-int} was derived 
  under the weaker condition 
  $\lim_{t\to\pm\infty}\|(S-\I t)^{-1}\|=0$
  but with the additional assumption that the integral in \eqref{eq:iR-int}
  exists for every $x\in X$.
  By the uniform boundedness principle, the projections $P_\pm$ are then bounded and $S$ is strictly dichotomous.
\end{remark}

\begin{remark}\label{rem:sectdichot}
An operator $S$ is called \emph{sectorially dichotomous} if it is 
dichotomous and $S|X_+$ and $-S|X_-$ are sectorial operators
with angle $\theta\leq\pi/2$.
Note that sectorial dichotomy implies bisectoriality
\cite[Lemma~2.12]{tretter-wyss}
as well as strict dichotomy.
A question asked in \cite{tretter-wyss} is the following:
Is every bisectorial and dichotomous
operator also sectorially dichotomous?
If we assume strict dichotomy, the answer is yes
since in this case $\pm S|G_\pm$ are sectorial by Theorem~\ref{theo:almbisect}
and strict dichotomy implies $G_\pm=X_\pm$.

The stronger assumption of strict dichotomy seems reasonable,
for otherwise the spaces $X_\pm$ are not unique.
Moreover the main theorems in
\cite{tretter-wyss} actually yield strictly dichotomous operators,
compare Remark~\ref{rem:theo-splitting}.
\end{remark}

\section{The subspaces $M_\pm$} 
\label{sec:Mpm}

Recall that the subspaces $G_\pm$ are $S$- and $(S-\lambda)^{-1}$-invariant (Theorem~\ref{theo:splitting}).
However, even if $S$ is densely defined, 
the restrictions of $S$ to  $G_\pm$ do not need to be densely defined,
see Example~\ref{ex:nodense}.
Let
\begin{equation}
  M_\pm:=\overline{\range(A_\pm)}
\end{equation}
and let $S_{M\pm}$ be the part of $S$ in $M_\pm$, i.e.\ $S$ is the restriction
of $S$ to $\mD(S_{M\pm})=\set{x\in\mD(S)\cap M_\pm}{Sx\in M_\pm}$.
The spaces $M_\pm$ have been introduced in \cite{BGK1986} where it is shown 
that $\mD(S_{M\pm})$ is dense in $M_\pm$ if $S$ is densely defined, cf. Lemma~\ref{lem:SMdense}.
In the following lemma we do not assume density of $\mD(S)$.
Note that in general the space $M_\pm$ is not invariant under $S$
(compare Theorem~\ref{theo:Mpm}), which is why we use the part of $S$ 
in $M_\pm$ instead of a simple restriction.
This is also reflected in the notational difference between $S_{M\pm}$
and $S|{G_\pm}$.

\begin{lemma}\label{lem:Mpm}
  Let $S(X\to X)$ be such that 
  \eqref{eq:resolvinc-h} and \eqref{eq:resolvbnd-h} hold. Then:
  \begin{enumerate}
  \item $M_\pm\subset G_\pm$.
  \item
    $M_\pm$ is $(S-\lambda)^{-1}$-invariant for $\lambda\in\varrho(S)$
    and
    $\mD(S_{M\pm})=S^{-1}(M_\pm)$.
  \item\label{enumi:Spm-resolv} 
    $M_\pm$ is $(S|G_\pm-\lambda)^{-1}$-invariant for 
    $\lambda\in\varrho(S|G_\pm)$,
    \begin{gather*}
      \sigma(S_{M\pm})=\sigma(S|G_\pm),\\
      (S_{M\pm}-\lambda)^{-1}x=(S|G_\pm-\lambda)^{-1}x,\quad x\in M_\pm,\,
      \lambda\in\varrho(S_{M\pm}).
    \end{gather*}
  \end{enumerate}
\end{lemma}

\begin{proof}
  (i) follows from $\range(A_\pm)\subset\ker(A_\mp)=G_\pm$ and the closedness of $G_\pm$.

  (ii) From $(S-\lambda)^{-1}A_\pm=A_\pm(S-\lambda)^{-1}$ it follows that
  $\range(A_\pm)$ and hence $M_\pm$ are 
  $(S-\lambda)^{-1}$-invariant. 
  In particular $S^{-1}(M_\pm)\subset M_\pm$, which implies
  $\mD(S_{M\pm})=S^{-1}(M_\pm)$.

  (iii)
  Let us prove the invariance of $M_+$.
  Recall that $\varrho(S|G_+)=\varrho(S)\cup\C_-$.
  For $\lambda\in\varrho(S)$ we have 
  $(S|G_+-\lambda)^{-1}=(S-\lambda)^{-1}|G_+$,
  so the invariance of $M_\pm$ follows from (ii) in this case.
  For $\lambda\in\C_-$, the invariance follows from \eqref{eq:resolv-G+}.
  The proof for $M_-$ is analogous.%

  The invariance property of $M_\pm$ immediately yields
  $\varrho(S|G_\pm)\subset\varrho(S_{M\pm})$ and
  $(S_{M\pm}-\lambda)^{-1}x=(S|G_\pm-\lambda)^{-1}x$ for $x\in M_\pm$ and
  $\lambda\in\varrho(S|G_\pm)$.
  Now
  \begin{equation*}
     \sigma(S)
     = \sigma(S|G_+) \,\dot{\cup}\, \sigma(S|G_-)
     \supset \sigma(S_{M+}) \,\dot{\cup}\, \sigma(S_{M-}).
  \end{equation*}
  So to prove $\sigma(S|G_\pm)\subset \sigma(S_{M\pm})$
  it suffices to show $\sigma(S)\subset \sigma(S_{M+}) \cup \sigma(S_{M-})$
  or, equivalently,
  $\varrho(S_{M+}) \cap \varrho(S_{M-})\subset\varrho(S)$.
  The proof is similar to the one for \eqref{eq:Apm-spec} in
  Lemma~\ref{lem:Apm}:
  If $\lambda\in\varrho(S_{M+})\cap\varrho(S_{M-})$ and $(S-\lambda)x=0$,
  then $x\in\mD(S^2)$ and hence
  \[x=y_++y_- \quad\text{with}\quad
  y_\pm= P_\pm x = A_\pm S^2x\in\range(A_\pm)\subset M_\pm.\]
  In fact $y_\pm\in\mD(S_{M\pm})$ since $x\in\mD(S^3)$ and 
  $Sy_\pm=A_\pm S^3x$.
  Consequently
  \[(S-\lambda)x=(S_{M+}-\lambda)y_++(S_{M-}-\lambda)y_-=0\]
  and thus $y_+=y_-=x=0$.
  The surjectivity of $S-\lambda$ is obtained by considering
  $T=(S_{M+}-\lambda)^{-1}A_++(S_{M-}-\lambda)^{-1}A_-$ and noting
  $(S-\lambda)S^2T=I$.
\end{proof}

The following result has been obtained as part of 
\cite[Theorem~3.1 and Lemma~3.3]{BGK1986}.
For the convenience of the reader we include the proof.

\begin{lemma}
   \label{lem:SMdense}
  If $S(X\to X)$ is densely defined and satisfies 
  \eqref{eq:resolvinc-h} and \eqref{eq:resolvbnd-h}, then
  $M_+\oplus M_-\subset X$ is dense and the operators
  $S_{M\pm}$ are densely defined.
\end{lemma}

\begin{proof} 
  From $S^{-2}=A_++A_-$ we obtain
  $\mD(S^2)\subset\range(A_+)\oplus\range(A_-)\subset M_+\oplus M_-$
  and thus the first assertion holds.
  For the second one note that
  $S^{-1}(\range (A_\pm))=A_\pm(\mD(S))$ is dense in $\range(A_\pm)$
  and hence in $M_\pm$,
  and that $S^{-1}(\range(A_\pm))\subset S^{-1}(M_\pm)=\mD(S_{M\pm})$.
\end{proof}

Despite the invariance of $M_\pm$  under $(S-\lambda)^{-1}$, 
and the invariance of $G_\pm$ under $S$ and $(S-\lambda)^{-1}$,
the subspaces $M_\pm$ are in general not invariant under $S$ itself; 
that is, the inclusion $S(\mD(S)\cap M_\pm)\subset M_\pm$ does not need to hold.
For densely defined operators $S$ the $S$-invariance of $M_\pm$
can be characterised as follows.

\begin{theorem}\label{theo:Mpm}
  Let $S(X\to X)$ be densely defined satisfying 
  \eqref{eq:resolvinc-h} and \eqref{eq:resolvbnd-h}. Then:
  \begin{enumerate} 
  \item The following equivalences hold:
    \begin{align*}
      &S|G_+ \text{ densely defined}
      \quad\Longleftrightarrow\quad M_+=G_+
      \quad\Longleftrightarrow\quad M_+ \text{ is $S$-invariant}, \\
      &S|G_- \text{ densely defined}
      \quad\Longleftrightarrow\quad M_-=G_-
      \quad\Longleftrightarrow\quad M_- \text{ is $S$-invariant}.
    \end{align*}
    In particular, $S_{M\pm}=S|G_\pm$ if $M_\pm=G_\pm$.
  \item If $P_+$ (or equivalently $P_-$) is bounded, then $M_\pm=G_\pm$.
  \item\label{enumi:Mpm_suff_Qn}
    Suppose there exist $Q_n\in L(X)$, $n\in\N$, such that\ \
    $\range(Q_n)\subset\mD(S)$,
    \begin{equation*}
      Q_n(S-\lambda)^{-1}=(S-\lambda)^{-1}Q_n\quad
      \quad\text{for all}\ \lambda\in\varrho(S)
    \end{equation*}
    and
    \begin{equation*}
      Q_nx\to x \quad\text{for all}\ x\in X.
    \end{equation*}
    Then $M_\pm=G_\pm$.
  \end{enumerate}
\end{theorem}

\begin{proof}
  (i)
  From $G_+=\range(P_+)$ and $P_+=S^2A_+$ we obtain 
  \begin{equation}\label{eq:G+M+incl}
     \mD(S^2)\cap G_+ = S^{-2}(G_+) = \range(A_+)\subset M_+.
  \end{equation}
  Suppose first that $S|G_+$ is densely defined.
  Since $0\in\varrho(S|G_+)$, $(S|G_+)^2$ is densely defined
  too, i.e.\ $\mD(S^2)\cap G_+\subset G_+$ is dense.
  Taking closures in \eqref{eq:G+M+incl}, we thus obtain $G_+\subset M_+$
  and hence $G_+=M_+$.
  Now let $M_+$ be $S$-invariant. From \eqref{eq:G+M+incl} we see
  $S^{-2}(G_+)\subset \mD(S^2)\cap M_+$,
  hence, by the $S$-invariance of $M_+$, we obtain
  $G_+\subset M_+$, i.e.\ $G_+=M_+$. 
  The other implications are trivial and the case of $M_-$, $G_-$ is analogous.

  (ii)
  If $P_\pm$ are bounded, then $X=G_+\oplus G_-$.
  Since $M_\pm$ are closed, it follows that $M_+\oplus M_-$ is closed.
  By Lemma~\ref{lem:SMdense} it is dense in $X$, so we obtain
  $M_+\oplus M_-=X$ and hence $M_\pm=G_\pm$.

  (iii)
  Since $Q_n$ commutes with the resolvent, we have
  $Q_nA_\pm=A_\pm Q_n$ and consequently
  $Q_nx \in \ker(A_-) = G_+$ for any $x\in G_+=\ker(A_-)$.
  Since additionally $Q_nx\in\mD(S)$ and $Q_nx\to x$, we obtain that 
  $G_+\cap\mD(S)$ is dense in $G_+$ and thus $M_+=G_+$ by (i). 
  The proof of $M_-=G_-$ is analogous.
\end{proof}

\begin{remark}\label{rem:Mpm-ortheig}
  The conditions of Theorem~\ref{theo:Mpm}\eqref{enumi:Mpm_suff_Qn}
  hold for example
  if $X$ is a Hilbert space which
  can be decomposed into an orthogonal direct sum of 
  finite-dimensional subspaces $X_k$ where each $X_k$ is spanned by
  a set of eigenvectors of $S$.
  Then $Q_n$ can be chosen as the orthogonal projection onto the subspace
  $X_1\oplus\dots\oplus X_n$.
  Block diagonal operators as in Example~\ref{ex:unbproj} admit such
  orthogonal decompositions.
\end{remark}

Although for a general densely defined operator $S$ its restrictions $S|G_\pm$ may fail to be densely defined, they always are if $S$ is bisectorial with $0\in\varrho(S)$.
\begin{lemma}\label{lem:bisect-Spm}
  Let $S(X\to X)$ be densely defined and bisectorial with $0\in\varrho(S)$.
  Then $S|G_\pm$ are densely defined and $M_\pm=G_\pm$.
\end{lemma}
\begin{proof}
  Since $S$ is densely defined and $\|\I t(\I t-S)^{-1}\|$ is bounded for 
  $t\in\R$, we have
  \[x=\lim_{t\to\infty}\I t(\I t-S)^{-1}x, \qquad x\in X,
  \]
  see \cite[Proposition~2.1.1]{haase}.
  Moreover, $(\I t-S)^{-1}x\in G_\pm\cap\mD(S)$ for $x\in G_\pm$, therefore
  $S|G_\pm$ are densely defined and hence $M_\pm=G_\pm$ by Theorem~\ref{theo:Mpm}.
\end{proof}

\begin{remark}
  If $X$ is reflexive, 
  then every sectorial (or bisectorial) operator is automatically 
  densely defined,
  see e.g.\ \cite[Proposition~2.1.1]{haase}.
  From Theorem~\ref{theo:almbisect} we already know that 
  if $S$ is bisectorial, then the restrictions $\pm S|G_\pm$
  are sectorial; so they are densely defined if $X$ is reflexive.
  The previous lemma ensures that this is true also in 
  non-reflexive spaces.
\end{remark}

Finally we show that the spaces $M_\pm$ can be expressed in terms of the operators $B_\pm$ from \eqref{eq:Bpmdef} if $S$ is densely defined and almost bisectorial.
\begin{lemma}
  Let $S(X\to X)$ be densely defined and almost bisectorial.
  Let $A_\pm$ and $B_\pm$
  be the operators defined in \eqref{eq:Apmdef} and \eqref{eq:Bpmdef}.
  Then
  \[M_\pm=\overline{\range(A_\pm)}=\overline{\range(B_\pm)}.\]
\end{lemma}
\begin{proof}
  Since $B_\pm$ is bounded and $\mD(S)$ is dense,
  \begin{equation*}
     \overline{\range(A_\pm)} = \overline{\range(B_\pm S^{-1})}
     = \overline{\range(B_\pm|\mD(S))} = \overline{\range(B_\pm)}.
     \qedhere
  \end{equation*}
\end{proof}

\section{Perturbation results} 
\label{sec:pert} 

Our first perturbation result generalises 
\cite[Theorem~5.1]{BGK1986} where the stronger assumptions
$\eps=1$, $\mD(T^2)\subset\mD(S^2)$ and exponential dichotomy of $S$ were required.

\begin{theorem}\label{theo:dichot-pert}
  Let $S(X\to X)$ be a densely defined and
  strictly dichotomous operator on the Banach space $X$.
  Suppose that $T(X\to X)$ is densely defined
  and that there exist $h>0$, $\eps>0$
  such that the following conditions hold:
  \begin{enumerate}
  \item\label{enumi:dichot-pert:ri}
    $\set{\lambda\in\C}{|\re\lambda|\leq h}\subset\varrho(S)\cap\varrho(T)$;
  \item\label{enumi:dichot-pert:es}
    $\sup_{|\re\lambda|\leq h}|\lambda|^{1+\eps}
    \|(S-\lambda)^{-1}-(T-\lambda)^{-1}\|<\infty$;
  \item\label{enumi:dichot-pert:de}
    $\mD(S^2)\cap\mD(T^2)\subset X$ dense.
  \end{enumerate}
  Then $T$ is strictly dichotomous too.
\end{theorem}
\begin{proof}
  Since $S$ is strictly dichotomous, 
  $(S-\lambda)^{-1}$ is uniformly bounded
  for $|\re\lambda|\leq h'$ by \eqref{eq:resolvbnd-h0} for some $0<h'\leq h$.
  Condition \eqref{enumi:dichot-pert:es} implies
  that $(T-\lambda)^{-1}$ is also uniformly bounded for $|\re\lambda|\leq h'$.
  Let $P_+^S$, $P_+^T$ be the projections corresponding to the spectrum
  in $\C_+$ of $S$ and $T$, respectively.
  Using \eqref{eq:Ppmint},
  \[\frac{1}{\lambda^2}(S-\lambda)^{-1}S^2
  =\frac{1}{\lambda^2}S+\frac{1}{\lambda}(S-\lambda)^{-1}S
  =\frac{1}{\lambda^2}S+\frac{1}{\lambda}+(S-\lambda)^{-1},\]
  and the respective identity for $T$, we obtain for
  $x\in\mD(S^2)\cap\mD(T^2)$,
  \begin{align*}
    P_+^Sx-P_+^Tx
    &=\frac{1}{2\pi\I}\int_{h'-\I\infty}^{h'+\I\infty}
    \!\biggl(\frac{1}{\lambda^2}(Sx-Tx)
    +\bigl((S-\lambda)^{-1}x-(T-\lambda)^{-1}x\bigr)\biggr)\,d\lambda\\
    &=\frac{1}{2\pi\I}\int_{h'-\I\infty}^{h'+\I\infty}
    \bigl((S-\lambda)^{-1}x-(T-\lambda)^{-1}x\bigr)\,d\lambda.
  \end{align*}
  By \eqref{enumi:dichot-pert:es} the last integral defines a bounded linear operator.
  Since $P_+^S$ is bounded, $P_+^T$ is bounded on the dense subspace $\mD(S^2)\cap\mD(T^2)$, and therefore $T$ is strictly dichotomous by Corollary~\ref{coroll:dichot}.
\end{proof}

\begin{remark}
  In \cite[Theorem~5.1]{BGK1986} it has been shown that,
  if $S$ is exponentially dichotomous
  then so is $T$.
  This implication remains true in our more general setting,
  where we require \eqref{enumi:dichot-pert:es} for some $\eps>0$
  instead of $\eps=1$, and \eqref{enumi:dichot-pert:de} instead
  of $\mD(T^2)\subset\mD(S^2)$.
  The proof of the exponential dichotomy of $T$ is largely identical to the
  one in \cite{BGK1986}.
\end{remark}

If the operator $S$ is almost bisectorial,
condition \eqref{enumi:dichot-pert:de}
of Theorem~\ref{theo:dichot-pert} can be relaxed:

\begin{theorem}\label{theo:dichot-pert-beta}
  Let $S(X\to X)$ be densely defined, almost bisectorial
  and strictly dichotomous.
  Let $T(X\to X)$ be a densely defined operator
  and $\eps>0$
  such that the following conditions hold:
  \begin{enumerate}
  \item\label{enumi:dichot-pert-beta:ri}
    $\I\R\subset\varrho(T)$;
  \item\label{enumi:dichot-pert-beta:es}
    $\sup_{\lambda\in\I\R}|\lambda|^{1+\eps}
    \|(S-\lambda)^{-1}-(T-\lambda)^{-1}\|<\infty$;
  \item\label{enumi:dichot-pert-beta:de}
    $\mD(S)\cap\mD(T)$ is dense in $X$.
  \end{enumerate}
  Then $T$ is also strictly dichotomous and almost bisectorial
  with the same exponent $\beta$ in the resolvent estimate
  \eqref{eq:almbisect}.
  In particular, if $S$ is bisectorial, then so is $T$.
\end{theorem}
\begin{proof}
  Condition \eqref{enumi:dichot-pert-beta:es} immediately implies that
  $T$ satisfies an estimate \eqref{eq:almbisect} with the same $\beta$
  as for $S$. 
  By Corollary~\ref{cor:iR-int} it suffices to show that
  \[\int_{-\I\infty}^{\I\infty\,\prime}
  (T-\lambda)^{-1}x\,d\lambda\]
  defines a bounded linear operator on $\mD(S)\cap\mD(T)$.
  Since $S$ is strictly dichotomous,
  the corresponding integral for $S$ is bounded on $\mD(S)$
  by Corollary~\ref{cor:iR-int}\,\eqref{enumi:iR-int-eq}.
  On the other hand, the difference of both integrals
  \[\int_{-\I\infty}^{\I\infty\,\prime}
  \bigl((S-\lambda)^{-1}-(T-\lambda)^{-1}\bigr)\,d\lambda\]
  converges in the uniform operator topology
  by \eqref{enumi:dichot-pert-beta:es}
  and thus defines a bounded linear operator.
\end{proof}

\begin{remark}
  It is especially the relatively weak condition \eqref{enumi:dichot-pert:de}
  which makes Theorem~\ref{theo:dichot-pert} 
  and~\ref{theo:dichot-pert-beta}
  more generally applicable than comparable theorems from
  \cite{BGK1986,LT2001,tretter-wyss} where 
  $\mD(T^2)\subset\mD(S^2)$ or $\mD(T)=\mD(S)$ was assumed.
  A situation where this generality is needed is the Hamiltonian operator
  matrix defined via extrapolation spaces in Example~\ref{ex:ham};
  in particular  $\mD(T)\neq\mD(S)$ there.
\end{remark}

We give an example showing that the unusual condition $\mD(T^2)\subset\mD(S^2)$
from \cite{BGK1986}
may fail even if $\mD(S)=\mD(T)$.

\begin{example}
   Let $S$ be an unbounded selfadjoint operator with strictly positive pure
   point spectrum, for instance the multiplication operator
   \begin{align*}
      S(l^2\to l^2),
      \qquad
      S(x_n)_{n\in\N} = (nx_n)_{n\in\N}
   \end{align*}
   with domain $\mD(S) = \{ (x_n)_{n\in\N} \in l^2 : (nx_n)_{n\in\N} \in l^2 \}
   \subsetneqq l^2$. 
   We take any $w\in l^2\setminus \mD(S)$ and define the bounded
   operator $R$ on $l^2$ by $Rx = P_w x$ where $P_w$ is the orthogonal
   projection onto $w$.  
   Set $T := S+R$.  Then $\mD(T) = \mD(S)$
   and $T$ is selfadjoint.
   On the other hand,
   \begin{align*}
      x \in \mD(S^2) \cap \mD(T^2)
      \; 
      &\implies\;
      x \in \mD(S^2)\, \wedge\, Tx = Sx + Rx \in \mD(T) = \mD(S)
      \\
      &\implies\;
      x \in \mD(S^2)\, \wedge\, Rx \in \mD(S).
   \end{align*}
   Since $\range (R) \cap \mD(S) = \{0\}$ by construction, it follows that
   $Rx=0$, hence $x\in \linspan\{w\}^\perp$.  
   Consequently
   $\mD(S^2)\cap \mD(T^2)\subset\linspan\{w\}^\perp$
   and so $\mD(S^2)\cap \mD(T^2)$ cannot be dense in $l^2$.
   Therefore
   \[\mD(S^2)\subsetneqq\mD(S^2)\cap \mD(T^2)
   \qquad 
   \text{and}
   \qquad 
   \mD(T^2)\subsetneqq\mD(S^2)\cap \mD(T^2)\]
   since $S^2$ and $T^2$ are densely defined,
   and we obtain $\mD(T^2)\not\subset\mD(S^2)$ as well as
   $\mD(S^2)\not\subset\mD(T^2)$.
   Note that $S$ and $T$ satisfy all conditions of 
   Theorem~\ref{theo:dichot-pert-beta}, but not
   \eqref{enumi:dichot-pert:de} from Theorem~\ref{theo:dichot-pert}.
\end{example}

One situation, where  condition 
\eqref{enumi:dichot-pert-beta:es} in Theorem~\ref{theo:dichot-pert-beta}
is fulfilled, is the case of
so-called $p$-subordinate perturbations.
The $p$-subordinate perturbations of bisectorial operators,
in particular the change of their spectrum,
have been studied in
\cite{tretter-wyss}.
\begin{definition}
  Let $S(X\to X)$, $R(X\to X)$ be linear operators.
  The operator $R$ is called \emph{$p$-subordinate} to $S$ with $0\leq p\leq1$ if
  $\mD(S)\subset \mD(R)$ and there exists $c>0$ such that
  \[\|Rx\|\leq c\|x\|^{1-p}\|Sx\|^p,\quad x\in\mD(S).\]
\end{definition}
\begin{corollary}
  Let $S(X\to X)$ be densely defined, almost bisectorial with exponent
  $\beta>1/2$ and strictly dichotomous.
  Let $R$ be $p$-subordinate to $S$ with $p<2\beta-1$ and let $T=S+R$.
  If $\I\R\subset\varrho(T)$, then $T$ is strictly dichotomous and almost bisectorial with the 
  same exponent $\beta$.
  Moreover, if $S$ is bisectorial, then so is $T$.
\end{corollary}
\begin{proof}
  First note that $\mD(T)=\mD(S)$ since $\mD(S)\subset\mD(R)$.
  Hence condition \eqref{enumi:dichot-pert-beta:de} in
  Theorem~\ref{theo:dichot-pert-beta} is satisfied
  and it remains to show that \eqref{enumi:dichot-pert-beta:es} holds too.
  Consider $\lambda\in\I\R$ and the identity
  \[T-\lambda=(I+R(S-\lambda)^{-1})(S-\lambda).\]
  Using $p$-subordination and almost bisectoriality, we get
  \begin{align*}
    \|R(S-\lambda)^{-1}\|
    &\leq c\|(S-\lambda)^{-1}\|^{1-p}\|S(S-\lambda)^{-1}\|^p\\
    &\leq c\frac{M^{1-p}}{|\lambda|^{(1-p)\beta}}
    \bigl(1+M|\lambda|^{1-\beta}\bigr)^p
    \leq \tilde{c}\frac{|\lambda|^{(1-\beta)p}}{|\lambda|^{(1-p)\beta}}
    =\frac{\tilde{c}}{|\lambda|^{\beta-p}}
  \end{align*}
  with $M$ as in \eqref{eq:almbisect},
  $\tilde{c}>0$ appropriate, and $|\lambda|$ large.
  Note that $p<2\beta-1\leq\beta$.
  Hence, for $\lambda\in\I\R$, $|\lambda|$ large,
  this implies $\lambda\in\varrho(T)$ and
  \[\|(T-\lambda)^{-1}\|\leq 2\|(S-\lambda)^{-1}\|
  \leq \frac{2M}{|\lambda|^\beta}.\]
  Consequently,
  \[\|(S-\lambda)^{-1}-(T-\lambda)^{-1}\|
  \leq\|(T-\lambda)^{-1}\|\|R(S-\lambda)^{-1}\|
  \leq\frac{2M\tilde{c}}{|\lambda|^{2\beta-p}},\]
  where $2\beta-p>1$.
\end{proof}
\begin{remark}
  In the bisectorial case, the previous result has essentially been obtained
  in \cite[Corollary~3.10]{tretter-wyss},
  with the assumption that $S$ is sectorially dichotomous
  and the conclusion that $T$ is dichotomous,
  compare Remarks~\ref{rem:theo-splitting} and~\ref{rem:sectdichot}.
\end{remark}

\section{Examples} 
\label{sec:examp}

\subsection{Non-uniqueness of the decomposition $X=X_+ \oplus X_-$}
\label{subsec:nonunique2} 
In Example~\ref{ex:nonunique} we saw that the decomposition of a
dichotomous operator is not necessarily unique.
The following example shows that this is even possible
for bisectorial operators on Hilbert spaces.

A linear operator $S(X\to X)$ on a Hilbert space $X$ is called \emph{accretive} if
\[
\C_-\subset\varrho(S)
\quad\text{and}\quad
\|(S-\lambda)^{-1}\|\leq\frac{1}{|\re\lambda|},\quad \re\lambda<0.\]
For instance, if $S$ is the generator of a nilpotent contraction semigroup on $X$, then $-S$ is accretive with $\sigma(S) = \varnothing$.

Every accretive operator $S$ has a square root $S^{1/2}$ which is sectorial with any angle $\theta>\pi/4$, see \cite[Proposition~3.1.2]{haase}.
In particular, $S^{1/2}$ is bisectorial.

\begin{example}\label{ex:nonunique2}
  Let $X$ be a Hilbert space and let $S(X\to X)$ be an accretive operator
  with $\sigma(S)=\varnothing$. 
  Then $S-\lambda^2=(S^{1/2}-\lambda)(S^{1/2}+\lambda)$
  shows that
  $\sigma(S^{1/2})=\varnothing$ too and, 
  as in Example~\ref{ex:nonunique},
  $S^{1/2}$ is dichotomous with respect to either of the 
  two decompositions
  \[X_+=X,\;\; X_-=\{0\}
  \qquad\text{and}\qquad
  X_+=\{0\},\;\; X_-=X.\]
  Here $S^{1/2}$ is strictly dichotomous only with respect to the 
  choice
  $X_+=X$, $X_-=\{0\}$.
\end{example} 

\subsection{An invertible bisectorial non-dichotomous operator}
\label{subsec:macintosh-yagi} 
McIntosh and Yagi \cite{mcintosh-yagi} gave the following example of an
invertible bisectorial operator that is not dichotomous.  We only sketch
their construction here.

\begin{example}\label{ex:mcintosh-yagi}
  Let $M>1$. 
  For every $m\in\N$, $m\geq 0$, choose $n\in\N$ such that
  \[\frac{M-1}{\pi\sqrt{18}}\log\left(\frac{n}{2}+1\right)\geq m\]
  and $(n+1) \times (n+1)$ matrices $D_m$ and $B_m$ as follows:
  $D_m$ is diagonal with entries $2^0,2^1,\dots,2^n$
  and $B_m$ is the Toeplitz matrix
  \[B_m=\frac{M-1}{\pi}
  \pmat{b_0&b_1&\dots&b_n\\
  b_{-1}&\ddots&\ddots&\vdots\\
  \vdots&\ddots&\ddots&b_1\\
  b_{-n}&\dots&b_{-1}&b_0}, \quad
  b_0=0,\;
  b_{\pm j}=\pm\frac{1}{j},\; j=1,\dots,n.\]
  Consider the block diagonal operator $A$ on $X=l^2$ given by
  \[A=\pmat{A_1&&\\&A_2&\\&&\ddots}, \qquad
  A_m=\pmat{D_m&B_mD_m\\0&-D_m}.\]
  It is then shown in \cite{mcintosh-yagi} that 
  $\sigma(A)\subset\,]-\infty,-1]\cup[1,\infty[\,$ and
  $\|(A-\lambda)^{-1}\|\leq M|\lambda|^{-1}$ for $\lambda\in\I\R\setminus\{0\}$.
  In particular, $A$ is a bisectorial operator.
  However, $A$ is \emph{not} dichotomous:
  In fact, the spectral projection $P_m$ corresponding to the positive
  eigenvalues of $A_m$ is
  \[P_m=\pmat{I&Z_m\\0&0}\]
  where the $(n+1)\times(n+1)$ matrix $Z_m$ satisfies
  $D_mZ_m+Z_mD_m=B_mD_m$ and $\|Z_m\|\geq m$.
  Since every dichotomous decomposition $X=X_+\oplus X_-$ must contain
  the eigenspace for $\lambda\in\C_\pm$ in $X_\pm$,
  the corresponding projection $P_+$ must contain all $P_m$ and thus 
  $P_+$ must be unbounded which contradicts the dichotomy assumption.
\end{example} 

\subsection{A densely defined operator $S$ with non-densely defined sectorial $S|G_\pm$}
\label{subsec:nondenseSG}
In this section we construct a densely defined
operator $S$ whose restriction $S|G_+$ to the positive spectral
subspace is not densely defined.
According to Theorem~\ref{theo:Mpm}, this is equivalent to $M_+\neq G_+$.
Our operator $S$ will be almost bisectorial, but not bisectorial, 
and its restriction $S|G_+$ will be sectorial.

Let $C([0,1])$ be the space of all continuous functions on $[0,1]$ and $C^1([0,1])$ the subspace consisting of all once differentiable functions on $[0,1]$ with continuous first derivative.

\begin{example}\label{ex:nodense} 
  On $X=C([0,1])$  consider the operator
  \[A_0f=f',\qquad
  \mD(A_0)=\set{f\in C^1([0,1])}{f(0)=0}.\]
  Then $A_0$ is non-densely defined, $\sigma(A_0)=\varnothing$, and
  $A_0$ is accretive.
  As in Section~\ref{subsec:nonunique2} it follows that
  $A_0$ has a square root, $A=A_0^{1/2}$,
  where $A$ is sectorial with any  angle $\theta>\pi/4$
  and $\sigma(A)=\varnothing$.
  Moreover, $A$ is non-densely defined too
  (otherwise $A_0=A^2$ had to be densely defined).
  In fact,
  \[\overline{\mD(A)}=\overline{\mD(A_0)}=\set{f\in X}{f(0)=0}.\]
  Let $w\in X$, $w(t)=1$ constant.
  Hence $w\in X\setminus\overline{\mD(A)}$
  and $X=\overline{\mD(A)}\oplus\linspan\{w\}$.
  Fix $0<s<1/2$ and consider the rank one operator
  $B:\mD(B)\subset l^2\to X$,
  \[B\alpha=\sum_{k=1}^\infty k^s\alpha_k \cdot w, \qquad
  \mD(B)=\biggset{\alpha = (\alpha_k)_{k=1}^\infty\in l^2}
     {\sum_{k=1}^\infty k^{s}|\alpha_k| <\infty}.\]
  Let $q\geq1$ and let
  $C:\mD(C)\subset l^2\to l^2$,
  \[C\alpha=(-k^q\alpha_k)_{k=1}^\infty, \qquad
  \mD(C)=\bigset{\alpha = (\alpha_k)_{k=1}^\infty\in l^2}
  {(-k^q\alpha_k)_{k=1}^\infty\in l^2}.\]

  \begin{lemma}
     The operators $B$ and $C$ are densely defined, 
     $C$ is selfadjoint, $\sigma(C) = \{-k^q : k\in\N\}$,
     $\mD(C)\subset\mD(B)$, 
     and $B$ is $1/q$-subordinate to $C$.
  \end{lemma}
  \begin{proof}
     The first assertions are clear.
     The last assertion follows because for all $\alpha\in\mD(C)$
     \begin{align*}
	&\|B\alpha\|\leq\sum_{k=1}^\infty k^s|\alpha_k|
	\leq\biggl(\sum_{k=1}^\infty \frac{1}{k^{2-2s}}\biggr)^{\frac12}
	\biggl(\sum_{k=1}^\infty k^2|\alpha_k|^2\biggr)^{\frac12},\\
	&\sum_{k=1}^\infty k^2|\alpha_k|^2\leq
	\biggl(\sum_{k=1}^\infty|\alpha_k|^2\biggr)^{1-\frac1q}
	\biggl(\sum_{k=1}^\infty k^{2q}|\alpha_k|^2\biggr)^{\frac1q}.
	\qedhere
     \end{align*}
  \end{proof}

  On $X\times l^2$ consider the operator
  \[S\pmat{f\\\alpha}=\pmat{A(f-B\alpha)\\C\alpha}, \qquad
  \mD(S)=\biggset{\pmat{f\\\alpha}\in X\times \mD(C)}{f-B\alpha\in\mD(A)}.\]

  \begin{proposition}
     The operator $S$ has the following properties:
     \begin{enumerate}
	\item $S$ is densely defined.
	\item $\sigma(S)=\sigma(C)$.
	\item $S$ is almost bisectorial.
	\item $G_+=X\times\{0\}$, $S|G_+$ is sectorial 
	  and $\mD(S|G_+)=\mD(A)\times\{0\}$.
	  In particular, $S|G_+$ is not densely defined.
     \end{enumerate}
  \end{proposition}
  \begin{proof}
     For the proof of (i) note that $\overline{\mD(A)}\times\{0\}\subset\overline{\mD(S)}$ and $(w,k^{-s}e_k)\in\mD(S)$ where $(e_k)$ is the standard orthonormal basis in $l^2$.
     This shows $(w,0)\in\overline{\mD(S)}$, hence
     $X\times\{0\}\subset\overline{\mD(S)}$.
     Finally, for every $x\in\mD(C)$, $(Bx,x)\in\mD(S)$ and thus 
     $(0,x)\in\overline{\mD(S)}$, so we showed that $S$ is densely defined.

     (ii) can be shown by direct computation. Moreover, for $\lambda\in\varrho(S)$,
     \begin{equation}\label{eq:nodense:resolv}
	(S-\lambda)^{-1}\pmat{f\\\alpha}=
	\pmat{(A-\lambda)^{-1}f+A(A-\lambda)^{-1}B(C-\lambda)^{-1}\alpha\\
	(C-\lambda)^{-1}\alpha}.
     \end{equation}

     On $\I\R$ the norms $\|\lambda(A-\lambda)^{-1}\|$ and $\|A(A-\lambda)^{-1}\|$
     as well as the analogous expressions for $C$ are uniformly bounded.
     The subordination property of $B$ yields the estimate
     \(\|B(C-\lambda)^{-1}\|\leq c|\lambda|^{-\beta}\),
     $\lambda\in\I\R\setminus\{0\}$, with
     $\beta=1-\frac1q$ and $c>0$.
     Therefore 
     \[\|(S-\lambda)^{-1}\|\leq\frac{M}{|\lambda|^\beta}, \qquad
     \lambda\in\I\R\setminus\{0\},\; \beta=1-\frac1q,\]
     with some $M>0$, which proves (iii).

     For the proof of (iv) let us first calculate $G_+$.
     Since for every $f\in X$,
     \[(S-\lambda)^{-1}\pmat{f\\0}=\pmat{(A-\lambda)^{-1}f\\0}\]
     is a bounded analytic function on $\overline{\C_-}$, we have
     $(f,0)\in G_+$.
     On the other hand, if $(f,\alpha)\in G_+$, then
     \eqref{eq:nodense:resolv} implies that $(C-\lambda)^{-1}\alpha$ has a
     bounded analytic extension to $\C_-$, which is possible only if $\alpha=0$.
     Thus 
     \[G_+=X\times\{0\},\]
     and therefore $\mD(S|G_+)=\mD(A)\times\{0\}$.
     It follows that $S|G_+$ is not densely defined and $S|G_+\cong A$, so $S|G_+$ is sectorial.
  \end{proof}

\end{example} 

\subsection{A non-densely defined non-sectorial but almost sectorial operator}
\label{subsec:nondenseSG2}
As a simple example of a non-densely defined operator on
a Hilbert space that satisfies an estimate \eqref{eq:almsect} with 
$\beta<1$ but is not sectorial,
we consider now an ordinary differential operator on the 
Sobolev space $H^1([0,1])$.
In a Banach space setting, operators of this form
have been considered in
\cite{silchenko-sobolevskii}.
A non-densely defined operator which is even sectorial and defined on
a non-reflexive Banach space is the operator $A$ from Example~\ref{ex:nodense}.
\begin{example}\label{ex:rdecbeta}
  For $m\geq 1$
  consider the operator $A_0$
  on $L^2([0,1])$ given by
  \begin{align*}
    &A_0f=(-1)^mf^{(2m)}, \\
    &\mD(A_0)=\set{f\in H^{2m}([0,1])}{f^{(j)}(0)=f^{(j)}(1)=0,\,
    j=0,\dots,m-1}.
  \end{align*}
  The operator $A_0$ is  positive selfadjoint.
  Let $A$ be the part of $A_0$ on the Sobolev space $H^1([0,1])$, i.e.,
  \begin{align*}
    &Af=A_0f=(-1)^mf^{(2m)},\\
    &\mD(A)=\set{f\in H^{2m+1}([0,1])}{f^{(j)}(0)=f^{(j)}(1)=0,\,
    j=0,\dots,m-1}.
  \end{align*}
  We easily see that $A$ is closed and $\sigma_p(A_0)=\sigma_p(A)$.
  Moreover, if $\lambda \in\varrho(A_0)$ then $A-\lambda$ is bijective
  with inverse
  $(A-\lambda)^{-1}=(A_0-\lambda)^{-1}|H^1([0,1])$
  and hence $\sigma_p(A)=\sigma(A)=\sigma(A_0)$.

  Since $A_0$ is selfadjoint, it is sectorial with arbitrary small angle $\theta>0$.
  Let $g\in L^2([0,1])$, $|\arg\lambda|\geq\theta>0$ and set
  $f=(A_0-\lambda)^{-1}g$.
  Then the identity $(-1)^mf^{(2m)}=\lambda f+g$
  yields
  \[\|f^{(2m)}\|_{L^2}\leq \big( 1+|\lambda|\|(A_0-\lambda)^{-1}\| \big)\|g\|_{L^2}
  \leq(1+M)\|g\|_{L^2}\]
  where $M>0$ is the constant from the sectoriality estimate
  \eqref{eq:sect}.
  Since $|f|_{2m}=\|f\|_{L^2}+\|f^{(2m)}\|_{L^2}$ is an equivalent norm
  on $H^{2m}([0,1])$
  and $\|(A_0-\lambda)^{-1}\|$ is uniformly bounded for $|\arg\lambda|\geq\theta$,
  this implies
  \[\|(A_0-\lambda)^{-1}g\|_{H^{2m}}\leq c_1\|g\|_{L^2},
  \qquad |\arg\lambda|\geq\theta,\]
  with $c_1>0$ depending on $\theta>0$.
  Using the interpolation inequality
  $\|f\|_{H^{k}}\leq c_2\|f\|_{L^2}^{1-k/n}\|f\|_{H^n}^{k/n}$,
  we obtain for $g\in H^1([0,1])$
  \begin{align*}
    \big\|(A_0-\lambda)^{-1}g\big\|_{H^1}
    &\leq c_2
    \big\|(A_0-\lambda)^{-1}g\big\|_{L^2}^{1-\frac{1}{2m}}\
    \big\|(A_0-\lambda)^{-1}g\big\|_{H^{2m}}^{\frac{1}{2m}}\\[1ex]
    &\leq\frac{c_3}{ |\lambda|^{1-\frac{1}{2m}} }\, \|g\|_{L^2}
    \leq\frac{c_3}{ |\lambda|^{1-\frac{1}{2m}} }\, \|g\|_{H^1}
  \end{align*}
  with $c_3>0$ depending on $\theta>0$.
  Consequently $A$ is almost sectorial,
  \begin{equation}\label{eq:rdec-ex}
    \|(A-\lambda)^{-1}\|\leq\frac{c_3}{|\lambda|^\beta},
    \qquad |\arg\lambda|\geq \theta,\quad
    \beta=1-\frac{1}{2m}.
  \end{equation}
  Moreover $A$ even has a compact resolvent: indeed our calculations imply
  that $(A_0-\lambda)^{-1}$ is a bounded operator from $L^2([0,1])$ to
  $H^1([0,1])$ and the embedding $H^1([0,1])\hookrightarrow L^2([0,1])$
  is compact.
  Finally note that $A$ is not densely defined since the closure of
  $\mD(A)$ in $H^1([0,1])$ is $H^1_0([0,1])$. 
  In particular, the estimate \eqref{eq:rdec-ex} cannot be improved to
  $\beta=1$, i.e.,
  $A$ is not a sectorial operator because sectorial
  operators in reflexive spaces are densely defined.
  On the other hand, \eqref{eq:rdec-ex} can be improved to
  $\beta=1-\frac{1}{4m}$ as indicated in 
  \cite{silchenko-sobolevskii}.
\end{example}

\subsection{A densely defined operator $S$ with non-densely defined almost sectorial restriction $S|G_\pm$}
The following is a variant of Example~\ref{ex:nodense} in a Hilbert space
setting.
Here $S|G_+$ is non-densely defined and almost sectorial.

\begin{example}\label{ex:nodense-hilb}
  Let $A$ be the operator from Example~\ref{ex:rdecbeta} acting on
  $X=H^1([0,1])$. Then $\overline{\mD(A)}=H^1_0([0,1])$.
  Let $w_1(t)=t$, $w_2(t)=1-t$, so that
  \[X=\overline{\mD(A)}\oplus\linspan\{w_1,w_2\}.\]
  For $0<s<1/2$ define $B:\mD(B)\subset l^2\to X$ by
  \begin{align*}
    &B\alpha=\sum_{k=1}^\infty (2k)^s\alpha_{2k} \cdot w_1
    +\sum_{k=1}^\infty (2k-1)^s\alpha_{2k-1} \cdot w_2, \\
    &\mD(B)=\biggset{\alpha=(\alpha_k)_{k=1}^\infty\in l^2}
       {\sum_{k=1}^\infty k^{s}|\alpha_k| <\infty}.
  \end{align*}
  As in Example~\ref{ex:nodense} consider also the selfadjoint operator
  $C(\alpha_k)=(-k^q\alpha_k)$ on $l^2$ and then 
  $S:\mD(S)\subset X\times l^2\to X\times l^2$,
  \[S\pmat{f\\\alpha}=\pmat{A(f-B\alpha)\\C\alpha}, \qquad
  \mD(S)=\biggset{\pmat{f\\\alpha}\in X\times \mD(C)}{f-B\alpha\in\mD(A)}.\]
  We obtain again that $B$ is $1/q$-subordinate to $C$, that $S$ is densely
  defined and that $(S-\lambda)^{-1}$ is given by \eqref{eq:nodense:resolv}
  where now $\sigma(S)=\sigma(A)\cup\sigma(C)$.
  Together with the estimate
  $\|(A-\lambda)^{-1}\|\leq c|\lambda|^{-\beta}$ on $\I\R$, 
  $\beta=1-\frac{1}{2m}$, 
  we then arrive at
  \[\|(S-\lambda)^{-1}\|\leq \frac{M}{|\lambda|^{\gamma}}, \qquad
  \lambda\in\I\R\setminus\{0\},\; \gamma=1-\frac{1}{2m}-\frac{1}{q},\]
  with some $M>0$.
  Theorems~\ref{theo:splitting} and~\ref{theo:almbisect}
  thus yield the $S$-invariant subspaces $G_\pm$.
  As in Example~\ref{ex:nodense}, we  derive
  $G_+=X\times\{0\}$ and $S|G_+\cong A$, which is not densely defined.
  Here $S|G_+$ is not sectorial but almost sectorial
  with $\beta=1-\frac{1}{2m}$.
\end{example}

\subsection{Hamiltonian operator matrices}
We can apply our theory to the Hamiltonian operator matrix appearing in 
systems theory in the case of so-called unbounded control and observation
operators.
We obtain criteria ensuring that the Hamiltonian is bisectorial and
strictly dichotomous.
Our setting generalises the ones in \cite{tretter-wyss}
and \cite{wyss-unbctrlham}
since we do not require a basis of (generalised) eigenvectors of the
Hamiltonian and at the same time
allow the control and observation operators to map 
out of the state space.

\begin{example}\label{ex:ham}
  Let $A$ be a sectorial operator on a Hilbert
  space $H$  with angle less than $\pi/2$ and $0\in\varrho(A)$.
  Let $H_1=\mD(A)$ be equipped with the graph norm and $H_{-1}$ the completion
  of $H$ with respect to the norm $\|A^{-1}\cdot\|$.
  Let $H_1^d$, $H_{-1}^d$ be the corresponding spaces for $A^*$.
  Moreover, we consider  intermediate spaces
  in the sense of Lions and Magenes
  \cite[Chapter~1]{lions-magenes}\footnote{
    This is equivalent to
    taking complex interpolation spaces, see
    \cite[Chapter~1, \S14]{lions-magenes}. Note that
    all involved spaces are Hilbert spaces.},
  \[H_s=[H_1,H]_{1-s},\quad H_{-s}=[H,H_{-1}]_{s},\quad s\in[0,1],\]
  Again, $H_s^d$ and $H_{-s}^d$ are defined analogously.
  In the special case when $A$ is selfadjoint,
  we obtain the fractional
  domain spaces $H_s=H_s^d=\mD(A^s)=\mD((A^*)^s)$,
  $s\in[-1,1]$,
  compare \cite[\S3]{wyss-unbctrlham}.
  In general however, $H_s\neq H_s^d$.
  The intermediate spaces yield bounded extensions
  \[A:H_{1-s}\longrightarrow H_{-s},\quad
  A^*:H^d_{1-s}\longrightarrow H^d_{-s},\quad s\in[0,1].\]
  Using the scalar product $(\cdot|\cdot)$ of $H$, we can identify
  the dual of $H_{-s}$ with $H_s^d$.
  This means that the scalar product of $H$
  extends to a sesquilinear form $(x|y)_{-s,s*}$,
  $x\in H_{-s}$, $y\in H_s^d$.
  Similarly, the dual of $H_s$ is identified with $H_{-s}^d$.
  This is also referred to as taking duality with respect to the pivot space
  $H$, see e.g.\ \cite[\S\S2.9, 2.10]{tucsnak-weiss}.

  Let $U,Y$ be Hilbert spaces and consider for some fixed $0<s<1/2$
  bounded operators
  $B:U\longrightarrow H_{-s}$, $C:H_s\longrightarrow Y$, the control and observation operators,
  respectively.
  With the above duality identifications,
  their adjoints are bounded operators $B^*:H_s^d\longrightarrow U$,
  $C^*:Y\longrightarrow H^d_{-s}$. The Hamiltonian is now given by the operator matrix
  \[T=\pmat{A&BB^*\\C^*C&-A^*}.\]
  Let $V_s=H_s\times H^d_s$ and $V=H\times H$. The Hamiltonian induces
  the bounded linear mapping $T:V_{1-s}\longrightarrow V_{-s}$, which we  consider,
  for the moment, as
  an unbounded
  operator on $V_{-s}$ with domain $V_{1-s}$.
  We use the decomposition
  \[T=S+R, \quad S=\pmat{A&0\\0&-A^*},\quad
  R=\pmat{0&BB^*\\C^*C&0},\]
  where $S:V_{1-s}\longrightarrow V_{-s}$ and $R:V_s\longrightarrow V_{-s}$ are bounded.
  Let $\lambda\in\I\R\setminus\{0\}$.
  In the following estimates, $c$ denotes
  a positive constant which is independent of $\lambda$, but may change
  from estimate to estimate.
  By the sectoriality of $A$, we have
  \[\|(A-\lambda)^{-1}\|\leq c|\lambda|^{-1},\quad
  \|A(A-\lambda)^{-1}\|\leq 1+|\lambda|\|(A-\lambda)^{-1}\|\leq c,\]
  and the second inequality implies
  \[\|(A-\lambda)^{-1}\|_{H\to H_1}\leq c,\quad
  \|(A-\lambda)^{-1}\|_{H_{-1}\to H}\leq c.\]
  Interpolation yields
  \[\|(A-\lambda)^{-1}\|_{H\to H_s}\leq c|\lambda|^{-(1-s)}, \quad
  \|(A-\lambda)^{-1}\|_{H_{-s}\to H}\leq c|\lambda|^{-(1-s)},\]
  and then
  \[\|(A-\lambda)^{-1}\|_{H_{-s}\to H_s}\leq c|\lambda|^{-(1-2s)}.\]
  Since an analogous estimate holds for $(A^*-\lambda)^{-1}$, we obtain
  \begin{equation}\label{eq:interpol}
    \|(S-\lambda)^{-1}\|_{V_{-s}\to V_s}\leq c|\lambda|^{-(1-2s)}.
  \end{equation}
  Consider the identity
  \[T-\lambda=(I+R(S-\lambda)^{-1})(S-\lambda).\]
  From \eqref{eq:interpol} we obtain
  \[\|R(S-\lambda)^{-1}\|_{V_{-s}\to V_{-s}}
  \leq\|R\|_{V_s\to V_{-s}}\|(S-\lambda)^{-1}\|_{V_{-s}\to V_s}
  \leq c|\lambda|^{-(1-2s)}\]
  and since $s<1/2$ we get that, for
  $\lambda\in\I\R$ and $|\lambda|$ large, 
  $I+R(S-\lambda)^{-1}$ is an isomorphism on $V_{-s}$;
  consequently
  $\lambda\in\varrho(T)$ with
  \[\|(T-\lambda)^{-1}\|_{V_{-s}\to V_{-s}}
  \leq c\|(S-\lambda)^{-1}\|_{V_{-s}\to V_{-s}}
  \leq c|\lambda|^{-1}.\]
  Moreover
  \[(T-\lambda)^{-1}-(S-\lambda)^{-1}=-(T-\lambda)^{-1}R(S-\lambda)^{-1}\]
  and hence
  \[\|(T-\lambda)^{-1}-(S-\lambda)^{-1}\|_{V_{-s}\to V_{-s}}
  \leq c|\lambda|^{-(2-2s)}.\]
  We consider now the part of $T$ in $V_s$, which we denote again by $T$.
  Then, similar to the above, the identity
  \[T-\lambda=(S-\lambda)(I+(S-\lambda)^{-1}R)\]
  and the estimate
  \[\|(S-\lambda)^{-1}R\|_{V_s\to V_s}\leq c|\lambda|^{-(1-2s)}\]
  yield
  \begin{gather*}
    \|(T-\lambda)^{-1}\|_{V_{s}\to V_{s}}
    \leq c|\lambda|^{-1},\\
    \|(T-\lambda)^{-1}-(S-\lambda)^{-1}\|_{V_{s}\to V_{s}}
    \leq c|\lambda|^{-(2-2s)},
  \end{gather*}
  for   $\lambda\in\I\R$, $|\lambda|$ large.

  For the rest of this example,
  we consider $T$ as an operator on $V=H\times H$, i.e., we take
  the part of $T$ in $V$.
  Applying interpolation to the above results, 
  we get that,
  if $\lambda\in\I\R$ and $|\lambda|$ large enough, then
  $\lambda\in\varrho(T)$ and
  \begin{gather}\label{eq:ham-bisect-est}
    \|(T-\lambda)^{-1}\|\leq c|\lambda|^{-1},\\
    \label{eq:ham-pert-est}
    \|(T-\lambda)^{-1}-(S-\lambda)^{-1}\|
    \leq c|\lambda|^{-(2-2s)}.
  \end{gather}

  Next we show that $T$ is bisectorial.
  In view of \eqref{eq:ham-bisect-est} it suffices to show that $\I\R\subset\varrho(T)$.
  This will be done using the same technique as in \cite[Lemma~4.5]{tretter-wyss}.
  Suppose $\I t\in \I\R$ is contained in the 
  approximate point spectrum $\sigma_{\mathrm{app}}(T)$, i.e.,
  there exists a sequence $(x_n,y_n)\in V_{1-s}$ such that
  $\|x_n\|^2+\|y_n\|^2=1$ and
  $(T-\I t)(x_n,y_n)\to0$ in $V$ as $t\to\infty$.
  This implies
  \begin{align*}
    &((A-\I t)x_n|y_n)_{-s,s*}+(BB^*y_n|y_n)_{-s,s*}\to 0, \\
    &(C^*Cx_n|x_n)_{-s*,s}-((A^*+\I t)y_n|x_n)_{-s*,s}\to 0.
  \end{align*}
  Summing both equations and taking the real part gives us
  \begin{equation}\label{eq:ham-speccalc}
    \|B^*y_n\|^2_U+\|Cx_n\|^2_Y=(BB^*y_n|y_n)_{-s,s*}+(C^*Cx_n|x_n)_{-s*,s}
    \to0.
  \end{equation}
  From $(T-\I t)(x_n,y_n)\to 0$ it follows that
  $x_n+(A-\I t)^{-1}BB^*y_n\to 0$ as a limit in $H$. Since 
  $(A-\I t)^{-1}B$ is bounded as an operator $U\to H$ and 
  $B^*y_n\to 0$ by \eqref{eq:ham-speccalc}, this yields $x_n\to 0$.
  Analogously, using $Cx_n\to0$, we get $y_n\to 0$, which is a
  contradiction to $\|x_n\|^2+\|y_n\|^2=1$.
  Therefore $\sigma_{\mathrm{app}}(T)\cap\I\R=\varnothing$.
  Since $\partial\sigma(T)\subset\sigma_{\mathrm{app}}(T)$ and 
  $\I\R\cap\varrho(T)\neq\varnothing$ we conclude that
  in fact $\I\R\subset\varrho(T)$.

  We can now invoke Theorem~\ref{theo:dichot-pert-beta} to show the strict 
  dichotomy
  of $T$. Indeed the assumptions on $A$ imply that $S$ is bisectorial
  and strictly dichotomous. Moreover, by $\I\R\subset\varrho(T)$ and 
  \eqref{eq:ham-pert-est} the conditions 
  \eqref{enumi:dichot-pert-beta:ri} and
  \eqref{enumi:dichot-pert-beta:es}
  of Theorem~\ref{theo:dichot-pert-beta} are satisfied too.
  Hence if $\mD(S)\cap\mD(T)=V_1\cap\mD(T)\subset V$ is dense,
  then $T$ is strictly dichotomous.
  Note that a typical setting in control theory is $H=L^2(\Omega)$ with
  $\Omega\subset\R^n$,
  $A$ is an elliptic differential operator
  and the control and observation operators act on
  traces of functions on the boundary of $\Omega$.
  In this case $\mD(S)\neq\mD(T)$ but 
  $C_0^\infty(\Omega)^2\subset \mD(S)\cap\mD(T)$, i.e.,
  $\mD(S)\cap\mD(T)$ is in fact dense in $V$.
\end{example}


\bigskip
\bigskip
\centerline{\bf Acknowledgements}
\medskip
  This work was partially supported by a research grant of the 
  ``Fachgruppe Mathematik und Informatik'' at the University of Wuppertal
  and FAPA No. PI160322022 of the Facultad de las Ciencias of the Universidad de Los Andes.
\bigskip

\bibliography{lit}
\bibliographystyle{alpha}

\end{document}